\documentclass[12pt]{amsart}
\usepackage{graphicx}
\usepackage{amsmath}
\usepackage{mathrsfs}
\usepackage[all]{xy}
\usepackage{amssymb}
\usepackage{mathrsfs}
\usepackage{amscd}
\usepackage{color}
\usepackage{amsthm}
\usepackage{tikz-cd}
\usepackage{url}

\makeatletter
\def\@tocline#1#2#3#4#5#6#7{\relax
  \ifnum #1>\c@tocdepth 
  \else
    \par \addpenalty\@secpenalty\addvspace{#2}%
    \begingroup \hyphenpenalty\@M
    \@ifempty{#4}{%
      \@tempdima\csname r@tocindent\number#1\endcsname\relax
    }{%
      \@tempdima#4\relax
    }%
    \parindent\z@ \leftskip#3\relax \advance\leftskip\@tempdima\relax
    \rightskip\@pnumwidth plus4em \parfillskip-\@pnumwidth
    #5\leavevmode\hskip-\@tempdima
      \ifcase #1
       \or\or \hskip 1em \or \hskip 2em \else \hskip 3em \fi%
      #6\nobreak\relax
    \hfill\hbox to\@pnumwidth{\@tocpagenum{#7}}\par
    \nobreak
    \endgroup
  \fi}
\makeatother

{\theoremstyle{definition}
\newtheorem{Def}{Definition}[section]
\newtheorem{Rem}[Def]{Remark}
\newtheorem{Ex}[Def]{Example}
}
\newtheorem{Prop}[Def]{Proposition}
\newtheorem{Thm}[Def]{Theorem}
\newtheorem{Lem}[Def]{Lemma}
\newtheorem{Cor}[Def]{Corollary}
\newtheorem{Fact}[Def]{Fact}
\newcommand{\ilim}[1][]{\mathop{\varinjlim}\limits_{#1}}

\newcommand{\R}{\mathbb{R}}

\newcommand{\Z}{\mathbb{Z}}
\newcommand{\Q}{\mathbb{Q}}
\newcommand{\ima}{{\rm Im}}
\newcommand{\fr}{{\rm Frac}}
\newcommand{\scrH}{\mathscr{H}}
\newcommand{\wt}{\widetilde}
\newcommand{\ol}{\overline}
\newcommand{\A}{\mathscr{A}}
\newcommand{\M}{\mathscr{M}}
\newcommand{\B}{\mathscr{B}}

\begin{document}
\normalfont
\title{On Berkovich double residue fields and birational models}
\author{Keita Goto}
\address{
Department of Mathematics, National Taiwan University,
Rm 547, 5F Astronomy Mathematics Building,
 No. 1, Section 4, Roosevelt Rd, Da’an District, Taipei City 10617, 
Taiwan
}
\email{k.goto.math@gmail.com}
\begin{abstract}

 Just as a residue field can be considered for a point of an algebraic variety, we can also consider a residue field for a point of a Berkovich analytic space. 
This residue field is a valuation field in the algebraic sense.
Then we can consider its residue field  as a valuation field.
  We call it the Berkovich double residue field at the point.

In this paper,  we consider  a point $x$ of the Berkovich analytification of an algebraic variety and identify the Berkovich double residue field at $x$ with the union of the residue fields at the center of $x$ in birational models.
Besides, we concretely compute the Berkovich double residue field for any quasi monomial valuation.

\end{abstract}
\maketitle
\tableofcontents

\newpage

\section{Introduction}

In this paper,  non-Archimedean fields are allowed to be trivially valued.
For a point $x$ of a Berkovich analytic space, we denote by $\scrH(x)$ the
complete valuation field given as the residue field at $x$ in the sense of Berkovich, and call it
the \emph{completed residue field}.
More precisely,
denote by $\A$  a Banach ring, by $\M(\A)$  the Berkovich spectrum of $\A$,
and by $|\cdot|_x: \A \to \R_{\geq 0}$  the  multiplicative seminorm
corresponding to $x\in \M(\A)$.
Then,
for each $x\in \M(\A)$, the completed residue field
$\scrH(x)$ is obtained as the completion of the fraction field of
$\A/ \mathfrak{p}_x$ with respect to the valuation on $\A/ \mathfrak{p}_x$ induced by the multiplicative seminorm $x$,
where $\mathfrak{p}_x:=\{f\in \A \ |\ |f|_x=0\}$.
In particular, there exists a canonical homomorphism $\A \to \scrH(x)$.
See the end of \S 2.1 for more detail.
Further, we denote by $\wt{\scrH(x)}$
the residue field of the valuation field $\scrH(x)$, which
 is first introduced in \cite{Berk90}, and
 call it the \emph{Berkovich double residue field} in accordance with \cite{Mat-text}.

Let $X$ be a variety over $k$.
Here
$k$ is a non-Archimedean field,  and a \emph{variety} over $k$ means a
separated integral scheme of finite type over $k$.
Our goal is to give an algebraic description of the Berkovich double residue fields at points on
$X^{\rm an}$, where $X^{\rm an}$ means the \emph{Berkovich analytification} of $X$.
It is well-known that
there exists a canonical map $$\pi_X: X^{\rm an}\to X$$
defined as follows:
For any  $x\in X^{\rm an}$, we can take 
a corresponding map of the form $|\cdot |_x: A \to \R_{\geq 0}$ for some open affine subset ${\rm Spec} A$  of $X$.
Here, $|\cdot |_x$ is a multiplicative seminorm.
Then we obtain a canonical homomorphism $A \to \scrH(x)$ in the same way as above.
It induces a canonical morphism $\psi_x: {\rm Spec}\scrH(x)\to X$.
Because ${\rm Spec}\scrH(x)$ is a singleton, the image of $\psi_x$ is a singleton contained in $X$.
Then  $\pi_X(x)$ is defined as the point of $\psi_x ({\rm Spec}\scrH(x))$.

\begin{Def}[$=$ Definition \ref{model}]
    In the above setting, a \emph{model} $\mathcal{X}$ of $X$ is a separated flat integral $k^\circ$-scheme
    of finite type equipped with a datum of an isomorphism $\mathcal{X}\times_{{\rm Spec}k^\circ}{\rm Spec}k\cong X$,
    where $k^\circ$ is the valuation ring of $k$.
\end{Def}

By definition, we obtain a natural morphism $X\to \mathcal{X}$ for any model $\mathcal{X}$ of $X$.
In addition,
 we obtain a morphism ${\rm Spec}\scrH(x)\to X \to \mathcal{X}$ by composition with the  canonical map $\psi_x: {\rm Spec}\scrH(x)\to X$.
This morphism gives the following diagram.
\[
  \xymatrix{
    {\rm Spec} \scrH(x) \ar[r] \ar[d] & \mathcal{X} \ar[d]^{} \\
    {\rm Spec} \scrH(x)^\circ \ar[r]_{} \ar@{..>}[ru]_{} & {\rm Spec}k^\circ
  }
\]

This dotted arrow making the diagram commutative does not always exist. 
When it exists,  the \emph{center} of $x$ (in $\mathcal{X}$) is defined as
its image of the unique closed point of ${\rm Spec}\scrH(x)^\circ$, and denoted as $c_\mathcal{X}(x)$.
By the valuative criterion of separatedness
\cite{H}, such $c_\mathcal{X}(x)$ is uniquely determined if it exists.

Define $X^{\rm val}:=\pi_X^{-1}(\xi_X),$ where $\xi_X$ is the generic point of $X$.
Now consider a proper birational morphism $f: Y\to X$ such that, for some open set $U\subset X$, $f^{-1}(U)=U$ and
the restriction morphism $f: f^{-1}(U)\to U$ becomes the identity map $\mathrm{id}_U$.
Assume that there is a model $\mathcal{X}$ (resp. $\mathcal{Y}$) of $X$ (resp. $Y$) such that
there is a proper birational morphism $\tilde{f}: \mathcal{Y}\to \mathcal{X}$
 whose restriction over ${\rm Spec} k$ is the given birational morphism $f:Y\to X$.
Then
we can regard any $x\in X^{\rm val}$ as an element of $Y^{\rm val}(:=\pi_Y^{-1}(\xi_Y))$ via the birational map $f^{-1}: X \dashrightarrow Y$.
Indeed,  $x\in X^{\rm val}=\pi_X^{-1}(\xi_X)$ and $\xi_X\in U$ imply that $x\in U^{\rm an}$.
By assumption, $f^{-1}|_U:  U\to f^{-1}(U)$ is the identity map.
 Then $(f^{-1}|_U)^{\rm an}:  U^{\rm an}\cong (f^{-1}(U))^{\rm an} $ gives
an element $y:=(f^{-1}|_U)^{\rm an}(x)$ of $Y^{\rm an}$. More precisely, $y\in Y^{\rm val}$ holds.
This is how we identify $y\in Y^{\rm val}$ with $x\in X^{\rm val}$.
For convenience, denote $x \in Y^{\rm val}$ via the identification.
Here,
 $\tilde{f}(c_\mathcal{Y}(x))=c_\mathcal{X}(x)$ follows.
This gives the canonical injection
$\kappa(c_\mathcal{X}(x))\hookrightarrow \kappa(c_\mathcal{Y}(x))$.
As we see later (\S 2 and \S 3),
there is a canonical injection $\kappa(c_\mathcal{X}(x)) \hookrightarrow \wt{\scrH(x)}$.
It
factors through the injection $\kappa(c_\mathcal{X}(x))\hookrightarrow \kappa(c_\mathcal{Y}(x))$.
That is,
 $\kappa(c_\mathcal{X}(x))\hookrightarrow \kappa(c_\mathcal{Y}(x)) \hookrightarrow \wt{\scrH(x)}$ holds.
Here, note that $\wt{\scrH(x)}$ does not depend on the ambient space $X$ in the sense that $f^{\rm an}$ induces an isomorphism $\wt{\scrH(x)}\cong \wt{\scrH(y)}$
 (cf.
Proposition \ref{nb}).

In understanding $X^{\rm an}$, it is often useful to consider models of $X$.
Such an idea first appeared in \cite{Ray72}.
In the paper, Raynaud proved that, if $k^\circ$ is a complete discrete valuation ring, then 
 any formal scheme $\mathfrak{X}$ locally of finite type over $k^\circ$
induces a rigid analytic space  $\mathfrak{X}^{\rm rig}$ by what is called Raynaud generic fiber.
Moreover,
 Raynaud also proved that
Raynaud generic fiber gives 
an equivalence of categories between a category of flat formal schemes of finite type over $k^\circ$
up to admissible blow-ups and a category of rigid analytic spaces over $k$ such that they are quasi-compact and quasi-separated.
Note that admissible blow-ups of $\mathfrak{X}$ take important roles in considering a $G$-topology of  $\mathfrak{X}^{\rm rig}$.
Let $\varpi$ be a uniformizing parameter of $k^\circ$.
Since any model of $X$ induces such a formal scheme by $\varpi$-adic completion,
any  model of $X$ induces a rigid analytic space by Raynaud generic fiber.
Moreover, it follows from \cite[Proposition 3.3.1]{Berk90} that any  model of $X$ induces a Berkovich analytic space.
Considering the equivalence of the categories,
it is natural to use models of $X$ to understand properties of $X^{\rm an}$.
In paricular, it is expected that local properties of $X^{\rm an}$ can be expressed by admissible blow-ups.
Actually, Theorem \ref{first}, discussed later, supports this expectation.

For any variety $X$ over a non-Archimedean field $k$ and any $x\in X^{\rm val}$, we define a directed set
$B(X,x)$  as follows: 
First of all, we take a Grothendieck universe $\mathcal{U}$ such that $X\in \mathcal{U}$.
Here, we admit \emph{the axiom of universes} through this paper. In particular, we assume that there is a 
Grothendieck universe $\mathcal{V}$ such that $\mathcal{U}\in \mathcal{V}$.
This assumption allows us to justify taking limits and colimits that run over any directed subset of  $\mathcal{U}$ in $\mathcal{V}$.
Let $B(X,x)$ be a $\mathcal{V}$-small category satisfying the following condition $(*)$, where the $\mathcal{V}$-smallness means being an element of $\mathcal{V}$, which means $B(X,x) \in \mathcal{V}$.
  \[
  (*):=
  \begin{cases}
  \underline{\text{If } k=k^\circ,} \\
 \mathrm{Ob}(B(X,x)):= \{ \text{all }
   k\text{-varieties}\ \mathcal{X} \text{ equipped with a datum of a }
   \\
   \text{proper birational morphism } f_\mathcal{X} :\mathcal{X}\to X 
   \text{ such that }  c_{\mathcal{X}}(x) 
   \text{  exists } \\
   \text{in }
   \mathcal{X}
   \text{ and, for some non-empty open }
   \text{subscheme }U\subset X, 
   f_\mathcal{X}^{-1}(U)
   \\=U  \text{ and } 
   f_\mathcal{X}|_{f_\mathcal{X}^{-1}(U)}=\mathrm{id}_U 
   \text{ hold}
   \} \cap \mathcal{U}, \\
     \text{and for any } \mathcal{X} \text{ and } \mathcal{Y}\in \mathrm{Ob}(B(X,x)),\\
  \mathrm{Hom}(\mathcal{Y},\mathcal{X}):=\{\text{proper birational morphisms } f: \mathcal{Y}\to \mathcal{X} \text{ over } \\
  X
  \text{ such that } f|_U=\mathrm{id}_U \text{ holds for some non-empty open}
  \\ \text{subscheme }U\subset X
  \}. 
  \\
      \underline{\text{If } k\neq k^\circ,} \\
    \mathrm{Ob}(B(X,x)):= \{
    \text{all proper models}\ \mathcal{X}\ \text{of } X \text{ over } k^\circ
    \}\cap \mathcal{U}, \\
  \text{and for any } \mathcal{X} \text{ and } \mathcal{Y}\in \mathrm{Ob}(B(X,x)),\\
  \mathrm{Hom}(\mathcal{Y},\mathcal{X}):=\{\text{proper birational morphisms } f:\mathcal{Y}\to \mathcal{X}\\
  \text{such that } f|_X=\mathrm{id}_X \text{ holds}\}. 
  \end{cases}
  \]
Then
$B(X,x)$ becomes a directed $\mathcal{V}$-small set. See \S 3 for the detail.
Under this setting,
we obtain the following result which is our first main result.

\begin{Thm}[$=$ Theorem \ref{union}]\label{first}
  Let $X$ be a variety over a non-Archimedean field $k$.
  For any $x\in X^{\rm val}$, we define the directed set $B(X,x)$ as above.
  If $B(X,x)\neq \emptyset$, then it follows that
  \[\wt {\scrH(x)} =\bigcup_{\mathcal{X}\in B(X,x)} \kappa(c_{\mathcal{X}}(x)) \cong \ilim[\mathcal{X}\in B(X,x)] \kappa(c_{\mathcal{X}}(x)),\]
  where $\bigcup_{\mathcal{X}\in B(X,x)} \kappa(c_{\mathcal{X}}(x))$ means
  the union of all $ \kappa(c_\mathcal{X}(x))$'s as subfields of $ \wt{\scrH(x)}$
  under the canonical injections $\kappa(c_\mathcal{X}(x)) \hookrightarrow \wt{\scrH(x)}$.
\end{Thm}

It asserts that  $\wt{\scrH(x)}$ can be regarded as the union of the residue fields of the center of $x$
in birational models.
To construct suitable birational models   is a  very difficult central problem in birational algebraic geometry or arithmetic geometry: for instance, log resolution, semistable reduction, minimal model, canonical model, Iitaka fibration, Mori fibration among others.
Morally speaking, one main feature of the Berkovich double residue field is that it is defined intrinsically in terms of purely non-archimedean world, without relying on good model construction, while it captures important information on birational models, as we show here.

\vspace{0.1in}
From now on, we focus on `quasi monomial valuation,'
which is a basic class of valuations.
Note that, in this paper, we adopt two definitions that are slightly different from each other.
(See Definitions \ref{triv qmo}, \ref{qmo} and Remark \ref{dif qmo}.)


The following is our second main result.

\begin{Thm}[$=$ Theorem \ref{smon}]\label{fsmon}
  Let $X$ be a variety over a trivially valued field $k$.
  For any quasi monomial valuation
   $x\in X^{\rm val}$,
the Berkovich double residue field $\widetilde{\mathscr{H}(x)}$ is finitely generated
   over $\kappa(c_X(x))$ as a field.
   Further there exists some blow-up
   $\pi : X'\to X$ such that
   \[\wt {\scrH(x)} \cong \kappa(c_{X'}(x)).\]
   Here, `quasi monomial valuation' is in the sense of Definition \ref{triv qmo}.
\end{Thm}
In addition, we give a concrete description as part of proving the theorem.

Our third main result (=Theorem \ref{quot}) states what happens to a Berkovich double residue field
when taking the quotient by a finite group $G$.

We also prove an analogous result to the last assertion in Theorem \ref{fsmon},
over a complete discrete valuation field (CDVF for short) as follows.

\begin{Thm}[$=$ Theorem \ref{quasi}]\label{fourth}
  Let $K$ be a CDVF, $R$ be the valuation ring of $K$, and
  $k$ be the residue field of $K$.
  Assume that the characteristic of $k$ is 0.
  Let $X$ be a smooth connected projective $K$-analytic space.

If $x$ is a quasi monomial valuation, then there exists an SNC model $\mathcal{X}$ of $X$ such that
 $\wt {\scrH(x)} \cong \kappa(c_{\mathcal{X}}(x))$.
 Here, `quasi monomial valuation' is in the sense of Definition \ref{qmo}.
\end{Thm}

This paper is organized as follows:
In $\S 2$, we recall terminology and facts about Berkovich analytic spaces
 and centers of multiplicative (semi)norms. 
From $\S 3$, we start to state our original results.
In $\S 3$, we state general properties of the Berkovich double residue field. In
particular, we prove Theorem \ref{first} asserting that the Berkovich double residue field can be written as  the direct limit of residue fields at the centers of birational models, and analyze when we need only one birational model for the expression.
In $\S 4$, we study quasi monomial valuations. In particular, we  prove Theorem \ref{fsmon}.
One notable idea behind the proof is to apply results in \S 3 to
  construct a suitable birational model.
In $\S 5$, we study $\wt{\scrH(x)}$ for $x\in X^{\rm val}$ whose center  in $X$ is a quotient singularity.
By considering group actions, we extend the previous result to quotient singularities. In particular,
we prove Theorem \ref{quot}.
In $\S 6$, we study the case when the base field is a CDVF. In particular, we
prove Theorem \ref{fourth}, as an application of
 our discussion in $\S4$.

 \begin{center}
   { \large{Acknowledgements}}
 \end{center}
 This paper is a revised version of my master thesis.
 I would like to thank Associate Professor Yuji Odaka, who is my 
 advisor, for a lot of suggestive advice and productive discussions.
And I would like to thank  Professor Mattias Jonsson and  Dr. Ryota Mikami for their helpful comments.
I am also grateful to an anonymous referee for his patience with my revision.
 This work is supported by JSPS KAKENHI Grant Number JP20J23401 and NSTC of Taiwan, with grant number 112-2123-M-002-005.

\section{Preliminaries}

This section is mainly based on \cite{Berk90}.
\subsection{Berkovich Spectra}
\

Let $A$ be a commutative ring with identity 1.
\begin{Def}\label{Def1}
  \
A \emph{seminorm} on $A$ is a function $|\cdot|:A\to \R_{\geq 0}$
possessing the following properties:

\begin{enumerate}
  \item $|0|=0$,
  \item $|1|\leq 1$,
  \item $|f-g|\leq |f|+|g|,$
  \item $|fg|\leq|f||g|,$

\end{enumerate}
for all $f,g\in A$.
Furthermore, a seminorm $|\cdot|$ on $A$ is called
\begin{itemize}
  \item
a \emph{norm}
if an equality $|f|=0$ implies $f=0$.
\item
\emph{non-Archimedean}
if $|f-g|\leq \max\{|f|,|g|\}$ \ for all $f,g\in A$.
\item
\emph{multiplicative}
if  $|1|=1$ and $|fg|=|f||g|$ \ for all $f,g \in A$.
\end{itemize}
Let $x$ be a multiplicative seminorm on $A$ which is also denoted by $|\cdot|$ as a function.
Then define $\mathfrak{p}_x:=\{f\in A\ |\ |f|=0\}$ which is a prime ideal of $A$.
\end{Def}

For each seminorm $|\cdot|$ on $A$, $|\cdot|$ is a \emph{norm} on $A$
if and only if the induced topology is Hausdorff.
To emphasize that $|\cdot|$ is a norm, we often denote by $||\cdot ||$ the norm $|\cdot|$.
Two norms $||\cdot ||$ and $ ||\cdot ||'$ on $A$ are called \emph{equivalent} if there exist $d_1, d_2>0$ such that $d_1 ||f ||\leq ||f ||' \leq d_2 ||f ||$ holds for any $f\in A$. 
A pair $(A,||\cdot ||)$ consisting of a commutative unital ring $A$ and a norm $||\cdot||$ on $A$
is called a \emph{normed ring}.

\begin{Def}

A \emph{Banach ring} $\A=(\A,||\cdot||)$ is a normed ring $\A$
that is complete with respect to its norm $||\cdot||$.
\end{Def}

\begin{Ex}
  We can regard any commutative unital ring $A$ as a Banach ring by equipping it with the \emph{trivial norm} $|\cdot|_0$
  defined as below.

For each $f\in A$,
\[
  |f|_0 := \begin{cases}
    1 & ({\rm if} \ f\neq 0) \\
    0 & ({\rm if} \ f=0)
  \end{cases}
\]
  The trivial norm is  non-Archimedean. Moreover $(A,|\cdot|_0)$ is complete.
  Hence this is a Banach ring.
In particular, when $A$ is a domain, the norm is multiplicative.
\end{Ex}

\begin{Def}
  A norm $|\cdot|$ is called a \emph{valuation} if it is multiplicative.
\end{Def}
As you can easily see from the above example, for any field $k$, the trivial norm $|\cdot|_0$ is a valuation.
Then $( k, |\cdot|_0 ) $ is called a \emph{trivially valued field}.

\begin{Ex}
  Recall the definition of DVR in the algebraic sense.
  A DVR $R$ has an (additive) discrete valuation $v :R\to \Z$.
  The DVR $R$ is called
  a \emph{complete DVR} if $R$ is a Banach ring with respect
  to the norm defined as $||\cdot ||:=e^{-v} : R \to \R_{\geq 0}$.
  The norm is multiplicative and non-Archimedean.
\end{Ex}

\begin{Def}
A Banach ring $(K,||\cdot ||)$ is called
  \begin{itemize}
    \item
    a \emph{Banach field} if $K$ is a field.
    \item
    a \emph{complete valuation field} if $(K,||\cdot ||)$ is a commutative Banach field whose norm is multiplicative.
    \item
    a \emph{non-Archimedean field} if $(K,||\cdot ||)$ is a complete valuation field whose norm is non-Archimedean.
  \end{itemize}
\end{Def}
By definition, a trivially valued field
 is also a non-Archimedean field.
Any complete valuation field is a valuation field in the algebraic sense.

For any complete valuation field $k=(k,|\cdot|)$, the \emph{value group} of $k$ is defined by
\[|k^\times| :=\{|f|\in\R \ |\ f\in k^\times (=k\setminus \{ 0\} ) \}
.
\]
Further, we set
\[
\sqrt{|k^\times| }:= \{ a\in \R_{\geq 0} \ |\ a^n \in |k^\times|
\ \
{\rm for\  some} \ \ n\in \Z_{>0}\} .
\]
Then
$|k^\times|$ is a $\Z$-module and
$\sqrt{|k^\times| }$ is a $\Q$-vector space.
In particular, it holds that $\sqrt{|k^\times| }\cong |k^\times| \otimes _\Z \Q$.

\begin{Def}
Let $(\A,||\cdot||)$ be a Banach ring.
A seminorm $|\cdot|$ on $\A$ is bounded if there exists $C>0$ such that
$|f|\leq C||f||$ for all $f\in\A$.
\end{Def}

Let $(\A,||\cdot||)$ be a Banach ring and $I$ be an ideal of $\A$.
Define the \emph{residue seminorm} on $\A/I$ as follows:
For any $f\in \A/I$,
\[ |f|:=\inf \left\{||g||\in \R_{\geq0} \ \middle| \ g\in \A,  f=g+I\in \A/I  \right\}.\]
This is a seminorm on $\A/I$.
Here, $I$ is said to be \emph{closed} if and only if
the residue seminorm is a norm on $\A/I$.
If that's the case, then $(\A/I,|\cdot|)$ becomes a Banach ring again.

Suppose that $(\A,||\cdot||_\A)$ and $(\B,||\cdot||_\B)$ are
Banach rings.
\begin{Def}
  Let $\varphi : \A\to\B$ be a ring homomorphism.
 The map $\varphi : \A\to\B$ is \emph{bounded} if there exists $C>0$ such that
  $ ||\varphi (f)||_\B \leq C||f||_\A $
for each $f\in\A$.
The map  $\varphi : \A\to\B$ is said to be \emph{admissible} if the residue seminorm of $\A/\ker \varphi$
  is equivalent to the restriction  of the norm $||\cdot||_\B$  to $\ima \varphi$
  under a canonical isomorphism $\A/\ker \varphi \cong\ima \varphi$.
\end{Def}

A bounded homomorphism is the most fundamental morphism between two Banach
rings. 
Therefore, unless otherwise noted, for two Banach rings $\A$ and $\B$, any map
$\varphi : \A\to\B$ will mean a bounded
homomorphism.
An admissible homomorphism is a bounded homomorphism that satisfies the
fundamental theorem on homomorphisms as Banach rings.

\begin{Def}[{\cite[$\S$ 1.2]{Berk90}}]
  Let $\A$ be a commutative Banach ring with identity.
  The \emph{spectrum} $\M(\A)$ is the set of all bounded multiplicative seminorms on $\A$
  provided with the weakest topology with respect to which all real valued functions on $\M(\A)$
  of the form $|\cdot|\mapsto|f|,\ f\in\A,$ are continuous.
\end{Def}
For any complete valuation field $k$, it follows from \cite[p.13]{Berk90} that $\M (k)$
is a singleton.
A bounded homomorphism $\varphi : \A\to\B$ between Banach rings induces a continuous map
$$\varphi^\sharp :\M(\B)\to \M(\A)$$ defined by
$|f|_{\varphi^\sharp(x)}:=|\varphi (f)|_x$
for all $f\in\A$ and  $x\in \M(\B)$,
where $|\cdot |_x$ (resp. $|\cdot |_{\varphi^\sharp(x)}$) is the corresponding seminorm to $x\in \M(\B)$
(resp. $\varphi^\sharp(x)\in \M(\A)$).

\begin{Thm}[{\cite[Theorem 1.2.1]{Berk90}}]
  Let $\A$ be a non-zero commutative Banach ring with identity.
  The spectrum $\M(\A)$ is a nonempty, compact Hausdorff space.
\end{Thm}

For $x\in\M(\A)$, define $\mathfrak{p}_x:=\{f\in\A\ |\ |f|_x=0\}$ in the same way as Definition \ref{Def1}.
This is a prime ideal of $\A$.
Then the multiplicative seminorm $x$ on $\A$ induces
a residue seminorm $\ol{x}$ on $\A/\mathfrak{p}_x$ that becomes a valuation.
In particular,
$|\ol{f}|_{\ol{x}}=|f|_x$ holds
for each $f\in \A$.
By abuse of language we  denote by $x$ the induced valuation $\ol{x}$.
The completion  $\scrH (x)$ of the fraction field of $\A/\mathfrak{p}_x$ with respect to
this valuation $x$ is a complete valuation field.
In particular,
there exists a canonical map $\A \to \scrH (x)$.
This $\scrH(x)$ is called the \emph{completed residue field} of $x$.
For each $x\in \M(\A)$, the \emph{rational rank} of $x$ is a number defined as
$\dim_\Q \sqrt{|\scrH(x)^\times|}$.

From now on, we define the \emph{(Berkovich) double residue field} $\wt{\scrH(x)}$.
The double residue field is first introduced in \cite{Berk90} without giving a name.
The way to call it is in accordance with
\cite{Mat-text}.

For $x\in\M(\A)$, we obtain the completed residue field $\scrH(x)$.
Then $\scrH(x)$ is a complete valuation field. Therefore,
\[\scrH(x)^{\circ}
:=\{f\in\scrH(x)\ |\ |f|_x\leq 1\}
\]
is its valuation ring and
\[\scrH(x)^{\circ \circ}
:=\{f\in\scrH(x)\ |\ |f|_x< 1\}\]
is its maximal ideal.
Hence,
\[\wt{\scrH(x)}:=\scrH(x)^{\circ}/\scrH(x)^{\circ\circ}
\]
is a field. We call
this $\wt{\scrH(x)}$ the double residue field of $x$, which is the residue field of the valuation field $\scrH(x)$.

In this paper, we compute $\wt{\scrH(x)}$ concretely
when $x$ is a `(quasi) monomial valuation'
(see Definitions \ref{m1}, \ref{m2}, \ref{triv qmo}, \ref{qmo}).

On the other hand,
$\wt{\scrH(x)}$ is computed concretely when $x$ is a point of
 the Shilov boundary of a strictly $k$-affinoid space
  (cf. \cite[Proposition 2.4.4]{Berk90})
 or $x$ is a point of the Berkovich affine line over an algebraically closed non-Archimedean field 
 (cf. \cite[\S 1.4.4]{Berk90} for points of type 1,2 and 3, and \cite[Proposition 2.3]{BR10} for any points).

For any non-Archimedean field $k$, we also define $k^\circ$, $k^{\circ\circ}$ and $\wt{k}$ in the same manner.

\subsection{Berkovich analytifications}
\

Now we review the construction of Berkovich analytification $X^{\rm an}$ for any scheme $X$
of locally finite type over a non-Archimedean field $k$ in the sense of Berkovich \cite{Berk90}.

At first, we define a Banach ring corresponding to a closed disc.
\begin{Def}\label{def afd}
Let $(k,|\cdot|)$ be a non-Archimedean field.
\begin{itemize}

\item$\A$ is a \emph{Banach $k$-algebra} if $\A$ is a Banach ring equipped with a
bounded homomorphism
$k\to \A$.
\item
Let $(\A,|\cdot|)$ be a Banach $k$-algebra and $n$ be a positive integer.
For $r_1,\dots,r_n>0$ and , define:
\begin{align*}
 &\A\{r_1^{-1}T_1,\dots,r_n^{-1}T_n\} \\
   &:= \left\{f=\sum_{I\in\Z_{\geq 0}^n} a_I
T^I\ \middle|\ a_I\in \A, \limsup_{|I|\to \infty}|a_I|r^I\to 0 \right\},
\end{align*}
where $I=(i_1,\dots,i_n),$ \ $|I|=i_1+\cdots +i_n,$\
\ $T^I=T_1^{i_1}\cdots T_n^{i_n}$
and $r^I=r_1^{i_1}\cdots r_n^{i_n}$. This is a commutative Banach ring
with respect to the valuation $||f||:=\max_I |a_I|r^I$.
For brevity, this algebra will also be denoted by $\A\{r^{-1}T\}$.
We often consider the case when $\A=k$. In particular, $k\{r^{-1}T\}$ is called a \emph{Tate algebra}.
\item
A Banach $k$-algebra $\A$ is called a \emph{$k$-affinoid algebra} if there exists
an admissible surjection
$k\{r^{-1}T\} \twoheadrightarrow \A$
.
\end{itemize}
\end{Def}
\begin{Rem}
  By definition, $\A\{r^{-1}T\}$ has a natural admissible injection
  $\A\hookrightarrow\A\{r^{-1}T\}$. Similarly, any non-zero $k$-affinoid algebra $\A$ has a natural injection $k\hookrightarrow \A$ via the surjection $k\{r^{-1}T\} \twoheadrightarrow\A$.
  If $\A$ is $k$-affinoid, then  $\A\{r^{-1}T\}$ is also $k$-affinoid.
$E(0,r):=\M(k\{r^{-1}T\})$ is an analogue of the complex closed disc
centered at the origin with radii $r=(r_1,\dots,r_n)$.
\end{Rem}
\begin{Ex}
  Suppose that the valuation on $k$ is trivial.
  If $r_i\geq1$ for all $1\leq i \leq n$, then $k\{r^{-1}T\}$ coincides with the
  polynomial ring $k[T_1,\dots,T_n]$.
  If $r_i<1$ for all $1\leq i \leq n$, then $k\{r^{-1}T\}$ coincides with the
  ring of formal power series $k[[T_1,\dots,T_n]]$.
\end{Ex}

\begin{Def}
  $X$ is \emph{$k$-affinoid space} if
   $X=\M(\A)$
  for some $k$-affinoid algebra $\A$.
\end{Def}
\begin{Ex}
  $k\{r^{-1}T\}$ is a typical example of a $k$-affinoid algebra.
  Moreover $E(0,r)=\M(k\{r^{-1}T\})$ which we saw above
  is a typical example of a $k$-affinoid space.

\end{Ex}

We will make use of the following proposition later.

\begin{Fact}[{\cite[Proposition 2.1.3]{Berk90}}]\label{closed}
  Any $k$-affinoid algebra is noetherian and all of its ideals are closed.
\end{Fact}

\begin{Def}
  Let $\A$ be a $k$-affinoid algebra.
  Let
  $f=(f_1,\dots,f_n)$ and $g=(g_1,\dots,g_m)$ be sequences of $\A$,
  and let $p=(p_1,\dots,p_n)$ and $q=(q_1,\dots,q_m)$ be sequences of positive real numbers.
  Then $\A\{p^{-1}f,qg^{-1}\}$ is defined as follows:
\begin{align*}
&\A\{p^{-1}f,qg^{-1}\} \\
&:=\A\{p^{-1}T,qS\}/(T_1-f_1,\dots,T_n-f_n,g_1S_1-1,\dots,g_mS_m-1).
\end{align*}
\end{Def}

In general, any $k$-affinoid algebra is noetherian and all of its ideals are closed.
In particular, it implies that
any quotient of $k$-affinoid algebra is again $k$-affinoid.
Therefore, $\A\{p^{-1}f,qg^{-1}\}$ is $k$-affinoid. Further,
there exists a natural morphism $\A\to \A\{p^{-1}f,qg^{-1}\}$ such that it induces a closed embedding
$\M(\A\{p^{-1}f,qg^{-1}\})\hookrightarrow \M(\A)$ of topological spaces.
Set $X=\M(\A)$.
Then,
\[X\{p^{-1}f,qg^{-1}\}:=\{x\in X\ |\ |f_i|_x\leq p_i, |g_j|_x\geq q_j, 1\leq i\leq n,
1\leq j\leq m\}
\]
 is a closed set of $X$ that
is identified with $\M(\A\{p^{-1}f,qg^{-1}\})$
through this closed embedding.
Such affinoid spaces of the form $X\{p^{-1}f,qg^{-1}\}$
are called \emph{Laurent domains} in $X$.

We will make use of the following proposition later.
\begin{Fact}[{\cite[$\S$ 2.2]{Berk90}}]\label{Lau}
Let $X$ be a $k$-affinoid space.
 Laurent domains that contain a point $x\in X$
  form a basis of closed neighborhoods of $x$.
\end{Fact}
 We have considered the $k$-affinoid space $X=\M(\A)$ just as a topological space so far.
However, $X$ also has a structure sheaf $\mathscr{O}_X$
for which $(X,\mathscr{O}_X)$ becomes a locally ringed space (cf. \cite{Berk90}).
Then
an open set
 $U\subset X$
is also regarded as a locally ringed space, which
   is called  a \emph{$k$-quasiaffinoid space}.

\vspace{0.1in}
Roughly speaking, $k$-analytic spaces in Berkovich's sense
are obtained by gluing  $k$-quasiaffinoid spaces together.
Besides,
they have structure sheaves defined
by gluing  structure sheaves of $k$-quasiaffinoid spaces together.

We now explain  concretely how to construct the Berkovich analytification:

Set $n\in \Z_{>0}$ and $X:={\rm Spec } k[T_1,\dots, T_n]$, where $ k[T_1,\dots, T_n]$ is a polynomial ring in $n$ variables over $k$.
Then
the Berkovich analytification of $X$ is given as
\[X^{\rm an}:= \bigcup_{r\in\R^n_{>0}}E(0,r)
=\bigcup_{r\in\R^n_{>0}}D(0,r),
\]
where $D(0,r)=\{x\in E(0,r)\ | \ |T_i|_x<r_i, 1\leq i \leq n\}$.
This
$D(0,r)$ is a $k$-quasiaffinoid space as an open set in $E(0,r)$.
Then
we give a topological structure to $X^{\rm an}$ such that
the natural immersion  $D(0,r)\hookrightarrow X^{\rm an}$ is open
for each $r\in\R^n_{>0}$.
Moreover,
for any $r,r'\in \R^n_{>0}$ satisfying $r'-r\in\R^n_{\geq0}$,
these two open immersions $D(0,r)\hookrightarrow X^{\rm an}$ and $D(0,r')\hookrightarrow X^{\rm an}$
 are compatible with
 the natural open immersion $D(0,r)\hookrightarrow D(0,r')$.
That is, the following diagram commutes.
\[
\xymatrix{
D(0,r)\ar[rr]^-{}\ar[dr]_-{}&&D(0,r')\ar[dl]^-{}\\
&X^{\rm an}\ar@{}[u]|{\circlearrowright}&
}
\]
It implies that $X^{\rm an}$ has a structure sheaf defined by
gluing the structure sheaves on each $D(0,r)$ for $r\in\R^n_{>0}$ together.

Next, set $A:= k[T_1,\dots,T_n]/I$ for some ideal $I$ of $k[T_1,\dots,T_n]$ and set $X:={\rm Spec}A$.
Then
the Berkovich analytification of $X$ is given as
\[
X^{\rm an}:=
\bigcup_{r\in\R^n_{>0}}\M(k\{r^{-1}T\}/I\cdot k\{r^{-1}T\})
=\bigcup_{r\in\R^n_{>0}} D'(0,r),
\]
where $$D'(0,r):=\{x\in \M(k\{r^{-1}T\}/I\cdot k\{r^{-1}T\}) \
|\ |T_i|_x<r_i, 1\leq i \leq n\}.$$
Then the structure sheaf of $X^{\rm an}$ is defined by gluing together the structure sheaves on each $D'(0,r)$.
In this way, we obtain the associated analytic space with an affine scheme of finite type over $k$.
Note that, from a set-theoretic point of view,
the set $X^{\rm an}$ is identified with the set consisting of
all multiplicative seminorms on $A$ whose restrictions to $k$
coincide with the equipped norm on $k$.

Finally, let $X$ be a scheme of locally finite type over $k$.
Then the Berkovich analytification $X^{\rm an}$
is obtained by gluing together the associated $k$-analytic spaces $U^{\rm an}$
for each affine open subscheme $U\subset X$ of finite type over $k$.
In particular, this construction of $X^{\rm an}$ guarantees that
$U^{\rm an}$ is an open set in $X^{\rm an}$.
Note that gluing of Berkovich analytic spaces for the general case uses what is called
G-topology while we omit the detail.

\begin{Prop}\label{afnbd}
  Let $k$ be a non-Archimedean field and $X$ be a scheme
  of locally finite type over $k$.
Then, for each $x\in X^{\rm an}$, there exists a $k$-affinoid neighborhood
  $V
  \subset X^{\rm an}$ of $x$.
\end{Prop}
\begin{proof}
Pick $x\in X^{\rm an}$.
We may assume that $X$ is an affine scheme of the form ${\rm Spec} A$, where $A:=k[T_1,\dots,T_n]/I$ as the above construction.
  Further, by the construction of $X^{\rm an}$, we may assume that $x\in D'(0,r)$ for some $r\in \R_{>0}^n$ under the same notation as the construction.
  Now we pick $r'\in \R_{>0}^n$ such that $r'-r\in \R_{>0}^n$.
  Then
  it holds that
  $$x\in D'(0,r)\subset \M(k\{r^{-1}T\}/I\cdot k\{r^{-1}T\})\subset D'(0,r')\subset X^{\rm an}.$$
Therefore, the point $x\in X^{\rm an}$ has the $k$-affinoid neighborhood of the form
  $$\M(k\{r^{-1}T\}/I\cdot k\{r^{-1}T\}).$$
\end{proof}

Under the same notation as the proof of Proposition \ref{afnbd}, 
we obtain a completed residue field $\scrH(x)$ for any $x\in D'(0,r)$.
Then
a canonical morphism $\psi_x: {\rm Spec} \scrH(x) \to X$ is given by a homomorphism
$$A\cong k[T_1,\dots,T_n]/I\to k\{r^{-1}T\}/I\cdot k\{r^{-1}T\} \to \scrH(x).$$
These $\scrH(x)$ and $\psi_x$ do not depend on the choice of $k$-affinoid neighborhood. (See Proposition \ref{nb}.)
Here, a canonical continuous map $\pi_X: X^{\rm an} \to X$ is defined by $(x\mapsto \mathrm{Im } \psi_x)$.
Note that  the image of $\psi_x$ becomes a singleton contained in $X$ since $ {\rm Spec} \scrH(x)$ is a singleton.
In particular, $x\in X^{\rm an}$ induces a canonical homomorphism $\psi_x^\sharp: \kappa (\pi_X(x)) \to \scrH(x)$.

On the other hand,
$\pi_X(x)$ can also be written down as follows:
For each affine open subscheme $U={\rm Spec}A\subset X$, the restriction
$\pi_X|_{U^{\rm an}}:U^{\rm an}\to U$ is defined as a map sending
a multiplicative seminorm $x$ on $A$ to a prime ideal of the form $\mathfrak{p}_x=\{f\in A\ |\ |f|_x=0\}$.

\vspace{0.1in}
The Berkovich analytification $X\mapsto X^{\rm an}$
 satisfies many properties
including  GAGA type theorems.
In particular, we will use of the following proposition later.

\begin{Fact}[{\cite[$\S$ 3]{Berk90}}]
  \label{un}
  Let $k$ be a non-Archimedean field.
For any morphism
$\varphi :X\to Y$
 between two schemes locally of finite type over $k$,
there exists a natural morphism $\varphi ^{\rm an} :X^{\rm an}\to Y^{\rm an}$ as
 $k$-analytic spaces such that the following diagram commutes.
\[
\xymatrix{  X^{\rm an} \ar@{..>}[r]^{\varphi ^{\rm an}} \ar[d]^{\pi_X} &  Y^{\rm an} \ar[d]^{\pi_Y} \\
  X \ar[r]^\varphi & Y
  }
\]
\end{Fact}

\subsection{Centers} \label{center}
\

Let $X$ be a variety over a non-Archimedean field $k$.
For $x\in X^{\rm an}$, we will define the center of $x$.
Before that, we define a model of $X$.
\begin{Def}\label{model}
  In the above setting, a \emph{model} $\mathcal{X}$ of $X$ is a separated flat integral $k^\circ$-scheme
  of finite type with the datum of an isomorphism $$\mathcal{X}\times_{{\rm Spec}k^\circ}{\rm Spec}k\cong X.$$

\end{Def}

\begin{Rem}\label{pro}
  When $X$ is a projective variety, we can construct a projective model of $X$ as follows:
  Consider a closed immersion $X\hookrightarrow \mathbb{P}_k^n$. Since $\mathbb{P}_k^n$ is
  an open set in $\mathbb{P}_{k^\circ}^n$, we can take $\mathcal{X}$ as the closure of $X$ in $\mathbb{P}_{k^\circ}^n$.
On the other hand, if $k=k^\circ$, a model of $X$ is uniquely determined
up to isomorphisms as $X$ itself.
\end{Rem}
Now we consider the canonical continuous map $\pi_X: X^{\rm an} \to X$.
For any $x\in X^{\rm an}$, the canonical homomorphism $\kappa (\pi_X(x))\hookrightarrow \scrH(x)$
corresponds to the canonical morphism $\psi_x: {\rm Spec}\scrH(x)\to X$ as in the paragraph after Proposition \ref{afnbd}.
Then, we obtain ${\rm Spec}\scrH(x) \stackrel{\psi_x}{\longrightarrow} X \to \mathcal{X}$.
This morphism gives the following diagram.
\[
  \xymatrix{
    {\rm Spec} \scrH(x) \ar[r] \ar[d] & \mathcal{X} \ar[d]^{} \\
    {\rm Spec} \scrH(x)^\circ \ar[r]_{} \ar@{..>}[ru]_{} & {\rm Spec}k^\circ
  }
\]
This dotted arrow does not always exist. When it exists, we define the \emph{center of $x$} (in $\mathcal{X}$) as
its image of the unique closed point of ${\rm Spec}\scrH(x)^\circ$.
It is denoted as $c_\mathcal{X}(x)$.
By the valuative criterion of separatedness \cite{H}, such $c_\mathcal{X}(x)$ is uniquely determined if it exists.

Moreover, the above diagram is factored as below.

\[
  \xymatrix{
    {\rm Spec}\scrH(x) \ar[r]\ar[d]& {\rm Spec}\ \kappa (\pi_X(x)) \ar[r] \ar[d] & \mathcal{X} \ar[d] \\
    {\rm Spec}\scrH(x)^\circ \ar[r] \ar@{..>}[rru]&{\rm Spec} R_x \ar[r] \ar@{..>}[ru]^{} & {\rm Spec} k^\circ
  }
\]
The above $R_x$ is defined by
\[R_x:=\{f\in \kappa(\pi_X(x)) \ |\ |f|_x\leq 1\}.
\]

\begin{Rem}
  When $\mathcal{X}$ is a proper $k^\circ$-scheme, the center of $x$ always exists for any $x\in X^{\rm an}$.
  It follows from the valuative criterion of properness.
\end{Rem}

Now we describe the center of $x$ in $\mathcal{X}$ more concretely.
Suppose that $c_\mathcal{X}(x)\in {\rm Spec}A\subset \mathcal{X}$, where ${\rm Spec}A$ is an open affine subscheme in $\mathcal{X}$.
 Then it holds that $A\to \scrH(x)$ factors through $A\to R_x$.
Since
$\kappa(\pi_X(x))$ is a valuation field
with respect to $x$,
$R_x$ is its valuation ring and
\[
\mathfrak{m}_x:=\{f\in \kappa(\pi_X(x)) \ |\ |f|_x\ < 1\}
\]
is its maximal ideal,
where $|\cdot|_x$ is the valuation on $\kappa(\pi_X(x))$ (and $\scrH (x)$) induced by $x\in X^{\rm an}$.
Then $c_\mathcal{X}(x)\in {\rm Spec}A$ is given
as the point corresponding to a prime ideal of the form
$$\{f\in A\ |\ |f|_x<1\} \in {\rm Spec}A.
$$

\vspace{0.1in}

Now we get back to the topic.
Set $X^{\rm val}:=\pi_X^{-1}(\xi_X),$ where
$\xi_X$ is the generic point of $X$.
In other words, $X^{\rm val}$ corresponds to the set of
all points in  $X^{\rm an}$ that are identified with
valuations on the function field $K(X)$ whose restriction to $k$ is the equipped valuation on $k$.
Hence, for any $x\in X^{\rm val}$ and any affine open subscheme $U \subset X$,
 it holds that $x\in U^{\rm val}$.
 Indeed, 
for any open affine subscheme $U={\rm Spec }A$ in $X$,
 $x$ is also a valuation on $A$
whose restriction to $k(\subset A)$ is exactly the equipped valuation on $k$.
In particular,
for any  birational map $f: Y\dashrightarrow X$ between two $k$-varieties,
we can regard any $x\in X^{\rm an}$ as an element of $Y^{\rm val}$.
Indeed, since 
any birational map $f: Y\dashrightarrow X$ induces an isomorphism $K(X)\to K(Y)$,
we can regard $x\in X^{\rm val}$ 
as an element of $Y^{\rm val}$ through the isomorphism.
Hence, we often denote $x\in X^{\rm val}$ by $x\in Y^{\rm val}$ to emphasize that.
Further, when  the center of $x\in X^{\rm val}$ in $\mathcal{X}$ exists and
$f: \mathcal{Y}\to \mathcal{X}$ is a proper birational morphism between two models,
 where $\mathcal{X}$ is a model of $X$ and $\mathcal{Y}$ is a model of $Y$,
we can apply the valuative criterion of properness to the following diagram.
 \[
   \xymatrix{
     {\rm Spec} K(X) \ar[r] \ar[d] & \mathcal{Y} \ar[d]^{f} \\
     {\rm Spec} R_x \ar[r]_{\varphi_{x,\mathcal{X}}} \ar@{..>}[ru]_{\varphi_{x,\mathcal{Y}}} & \mathcal{X}
   }
 \]
Then we obtain the unique morphism
 $\varphi_{x,\mathcal{Y}}:{\rm Spec} R_x\to \mathcal{Y}$.
We can identify this $\varphi_{x,\mathcal{Y}}$ with a morphism that appears when
we define $c_\mathcal{Y}(x)$
 as above
 for $x\in Y^{\rm val}$.
Since $\varphi_{x,\mathcal{X}}=f\circ \varphi_{x,\mathcal{Y}}$,
 we obtain
 $f(c_\mathcal{Y}(x))=c_\mathcal{X}(x)$
 by chasing the unique closed point of ${\rm Spec}R_x$.
 In particular, there exists an injection
 $
\kappa(c_\mathcal{X}(x))\hookrightarrow \kappa(c_\mathcal{Y}(x))$.
It means that the lifting of the center induces an extension of the
residue field of the center.

\section{Some basic properties of $\wt{\scrH (x)}$}

In this section, we  see general properties of the completed residue field $\scrH(x)$
and the double residue field $\wt{\scrH (x)}$.
Unless otherwise described, we assume that $k$ is a non-Archimedean field.

In $\S 2$, we defined the completed residue field $\scrH(x)$ for $x\in X=\M (\A)$.
According to the definition, $\scrH(x)$ seems to depend on the ambient $k$-affinoid space $X$.
Hence, we temporarily denote $\scrH(x)$ by $\scrH_X(x)$.
However, the following proposition asserts that the completed residue field does not depend on the choice of $k$-affinoid neighborhood of $x$.
Therefore, we often abbreviate $\scrH_X(x)$ as $\scrH(x)$.

\begin{Prop} \label{nb}
  Let $X$ be a $k$-affinoid space.
  For any $x \in X$ and any $k$-affinoid neighborhood $V$ of $x$ in $X$,
  there exists a canonical isomorphism $$\scrH_X(x)\stackrel{\simeq}{\longrightarrow} \scrH_V(x).$$
\end{Prop}

\begin{proof}

  It is an immediate consequence of \cite[Corollary 2.5.16]{Berk90}. 
\end{proof}

Next, we consider the case when $X$ is the Berkovich analytification of some
$k$-variety.

\begin{Prop}[{cf. \cite[Step (2) of the proof of Theorem 3.4.1]{Berk90}}]\label{Hx}
  Let $A$ be a finitely generated algebra over a non-Archimedean field $k$.
  For any $x\in ({\rm Spec}A)^{\rm an}$, set  $\mathfrak{p}_x =\{f\in A \ | \ |f|_x=0\}$ as in Definition \ref{Def1}.
  Then there is a canonical isomorphism
  \[\scrH(x)\cong \widehat{{\rm Frac}(A/\mathfrak{p}_x)},\]
  where the right-hand side is the completion of ${\rm Frac}(A/\mathfrak{p}_x)$
  with respect to the norm induced by $x$.
In other words, $\scrH(x)$ can be regarded as the completion of the
 residue field $\kappa (\mathfrak{p}_x)$ at
 $\mathfrak{p}_x\in {\rm Spec}A$ with respect to $x$.
\end{Prop}
\begin{proof}
Take a $k$-affinoid neighborhood $V$ of $x$.
Then there exists a closed embedding $\iota: V\hookrightarrow E(0,r)$ for some $r\in \R_{>0}$.
Here, the morphism $\iota$ is also a closed immersion, which means that the morphism $\iota$ corresponds to a coherent ideal sheaf of $\mathscr{O}_{E(0,r)}$.
By definition of $\scrH(x)$, it holds that
$\scrH_V(x)\cong \scrH_{E(0,r)}(\iota (x))$.
Hence,
we may assume $A=k[T_1,\dots,T_d]$, which does not change $\widehat{{\rm Frac}(A/\mathfrak{p}_x)}$.
Take a $k$-affinoid neighborhood of $x$ of the form $E(0,r)=\M(k\{r^{-1}T\})$.
Set $\A:=k\{r^{-1}T\}$.
By  Proposition \ref{nb}, it holds that
$\scrH(x)\cong \scrH_{E(0,r)}(x)$.
For this $x\in E(0,r)$, we can also consider $\mathfrak{p}_{\A,x}=\{f\in \A \ |\ |f|_x=0 \}$.
Further, since $\mathfrak{p}_{x}=\mathfrak{p}_{\A,x}\cap A$,
$A\hookrightarrow k\{r^{-1}T\}=\A$ induces
$\varphi: A/\mathfrak{p}_x\hookrightarrow \A/\mathfrak{p}_{\A,x}$.

For any $f \in \A/\mathfrak{p}_{\A,x}$,
we can take a lift of the form
$$
\tilde{f}=\sum_{I\in \Z_{\geq 0}^d}a_{I}T^I\in \A= k\{r^{-1}T\}
$$
as in Definition \ref{def afd}.
Then we set $$f_n :=\sum_{|I|\leq n}a_{I}T^I\in A
$$ for any $n\in \Z_{\geq 0}$.
Note that  $x$ induces a bounded valuation
 on $\A/\mathfrak{p}_{\A,x}$.
 It is also denoted by $x$.
Then, 
\[
|f-\ol{f_n}|_x=|\sum_{|I|> n}a_{I}\ol{T}^I|_x
\]
\[
\leq ||\sum_{|I|> n}a_{I}\ol{T}^I||_{\A/\mathfrak{p}_{\A,x}}
\leq ||\sum_{|I|> n}a_{I}T^I||_{\A}\to 0
\]
as $n\to \infty$, where $\ol{f_n}$ means the image of $f_n$ through a homomorphism $A \twoheadrightarrow A/\mathfrak{p}_x\hookrightarrow \A/\mathfrak{p}_{\A,x}$. 
Since $A\hookrightarrow \A$ is dense,
$A/\mathfrak{p}_x$ is also dense in $\A/\mathfrak{p}_{\A,x}$ through the isometry $\varphi$.
It implies that ${\rm Frac}(A/\mathfrak{p}_x)$ is also dense in ${\rm Frac}(\A/\mathfrak{p}_{\A,x})$.
Indeed, for any non-zero element $h\in {\rm Frac}(\A/\mathfrak{p}_{\A,x})$, we can take non-zero
elements $f,g \in \A/\mathfrak{p}_{\A,x}$ such that $h=f/g$.
Since $A/\mathfrak{p}_x$ is  dense in $\A/\mathfrak{p}_{\A,x}$,
we can take a sequence $\{f_i\}_{i\in \Z_{>0}}$ (resp. $ \{g_j\}_{j\in \Z_{>0}}$) in $A/\mathfrak{p}_x$ converging to 
$f$ (resp. $g$).
In particular, we may assume that $f_i$ and $g_j$ are non-zeros for any $i,j\in \Z_{>0}$.
Then we consider a sequence $\{f_m/g_m\}_{m\in \Z_{>0}}$.
Now we show that the sequence converges to $f/g$.
Take $\varepsilon \in \R$ such that $0<\varepsilon <|g|_x$.
Then $|g_m-g|_x<\varepsilon$ implies 
$|g_m|_x=|g|_x$ since $|\cdot |_x$ is non-Archimedean.
Hence,
for sufficiently large $m\in \Z_{>0}$, 
we may assume $|g_m|_x=|g|_x$.
Then
it holds that $$\left|\frac{f}{g} -\frac{f_m}{g_m}\right|_x= \frac{|fg_m-f_mg|_x}{|g|_x^2} \leq \frac{\max\{|f(g_m-g)|_x, |g(f-f_m)|_x\}}{|g|_x^2 }\to 0$$
as $f_m\to f$ and $g_m\to g$.
Hence ${\rm Frac}(A/\mathfrak{p}_x)$ is dense in ${\rm Frac}(\A/\mathfrak{p}_{\A,x})$.
It implies that
$$\widehat{{\rm Frac}(A/\mathfrak{p}_x)} \cong \widehat{{\rm Frac}(\A/\mathfrak{p}_{\A,x})}=\scrH_{E(0,r)}(x) \cong \scrH(x).$$
\end{proof}

\begin{Cor}\label{Hx residue}
Let $X$ be a variety over a non-Archimedean field $k$.
Set $\pi_X: X^{\rm an}\to X$ as before.
Then, for any $x\in X^{\rm an}$, it holds that
\[
\scrH(x)\cong \widehat{\kappa(\pi_X(x))},
\]
where the right-hand side is the completion of the
 residue field $\kappa (\pi_X(x))$ at $\pi_X(x)\in X$ with respect to $x$.
 In particular, this isomorphism preserves the valuations.
\end{Cor}

\begin{proof}
  It follows from Proposition \ref{Hx}.
\end{proof}

Let $\mathcal{X}$ be a model of a $k$-variety $X$.
Set $\pi_X:X^{\rm an}\to X$ as before, and
take a point $x\in X^{\rm an}$.
Let us recall the definitions of the center $c_\mathcal{X}(x)$
and the double residue field $\wt{\scrH(x)}$.
If $c_\mathcal{X}(x)$ exists,
there is the canonical injection
\[\kappa (c_\mathcal{X}(x))\hookrightarrow \wt{\scrH(x)}.\]
Let $\mathcal{Y}$ be a model of a $k$-variety $Y$ such that 
there exists a proper birational morphism $f:\mathcal{Y}\to \mathcal{X}$.
Further, let us assume that
the induced morphism $K(X)\to K(Y)$ becomes the identity map
and $x\in X^{\rm val}$.
Then, it holds that
\[
\kappa(c_\mathcal{X}(x))\hookrightarrow \kappa(c_\mathcal{Y}(x))\hookrightarrow \wt{\scrH(x)}.
\]
In general, $\kappa(c_\mathcal{Y}(x))\cong  \wt{\scrH(x)}$ does not necessarily hold.
However, in some situations,
we can obtain $\mathcal{Y}$ such that
$\kappa(c_\mathcal{Y}(x))\cong  \wt{\scrH(x)}$ by taking appropriate blow-up.
That is the main theme of this paper.

\vspace{0.1in}
In general, we obtain the following results.

\begin{Lem}\label{res}

  Let $X$ be a variety over a non-Archimedean field $k$ and
  set $\pi_X: X^{\rm an}\to X$ as before.
  Then, for any $x\in X^{\rm an}$, it holds that
  \[
  \wt {\scrH(x)} \cong\wt{\kappa(\pi_X(x))},
  \]
  where the right-hand side is the residue field of the
 valuation field $\kappa (\pi_X(x))$ with respect to $x$.
   In particular, when $x\in X^{\rm val}$, it holds that
   \[
   \wt {\scrH(x)} \cong\wt{K(X)},
   \]
   where $K(X)$ is the function field of $X$.
\end{Lem}
\begin{proof}

It immediately follows from Corollary \ref{Hx residue}.
\end{proof}

Let $X$ be a variety over a non-Archimedean field $k$. For any $x\in X^{\rm val}$, we now define a directed set
$B(X,x)$  as follows:
First of all, we take a Grothendieck universe $\mathcal{U}$ such that $X\in \mathcal{U}$.
Here, we admit \emph{the axiom of universes} as we explained in \S 1. In particular, we assume that there is a 
Grothendieck universe $\mathcal{V}$ such that $\mathcal{U}\in \mathcal{V}$.
This assumption allows us to justify taking limits and colimits that run over any directed subset of  $\mathcal{U}$ in $\mathcal{V}$.
Let $B(X,x)$ be a $\mathcal{V}$-small category satisfying the following condition $(*)$, where the $\mathcal{V}$-smallness means being an element of $\mathcal{V}$, which means $B(X,x)\in \mathcal{V}$.
  \[
  (*):=
  \begin{cases}
  \underline{\text{If } k=k^\circ,} \\
 \mathrm{Ob}(B(X,x)):= \{ \text{all }
   k\text{-varieties}\ \mathcal{X} \text{ equipped with a datum of a }
   \\
   \text{proper birational morphism } f_\mathcal{X} :\mathcal{X}\to X 
   \text{ such that }  c_{\mathcal{X}}(x) 
   \text{  exists } \\
   \text{in }
   \mathcal{X}
   \text{ and, for some non-empty open }
   \text{subscheme }U\subset X, 
   f_\mathcal{X}^{-1}(U)
   \\=U  \text{ and } 
   f_\mathcal{X}|_{f_\mathcal{X}^{-1}(U)}=\mathrm{id}_U 
   \text{ hold}
   \} \cap \mathcal{U}, \\
     \text{and for any } \mathcal{X} \text{ and } \mathcal{Y}\in \mathrm{Ob}(B(X,x)),\\
  \mathrm{Hom}(\mathcal{Y},\mathcal{X}):=\{\text{proper birational morphisms } f: \mathcal{Y}\to \mathcal{X} \text{ over } \\
  X
  \text{ such that } f|_U=\mathrm{id}_U \text{ holds for some non-empty open}
  \\ \text{subscheme }U\subset X
  \}. 
  \\
      \underline{\text{If } k\neq k^\circ,} \\
    \mathrm{Ob}(B(X,x)):= \{
    \text{all proper models}\ \mathcal{X}\ \text{of } X \text{ over } k^\circ
    \}\cap \mathcal{U}, \\
  \text{and for any } \mathcal{X} \text{ and } \mathcal{Y}\in \mathrm{Ob}(B(X,x)),\\
  \mathrm{Hom}(\mathcal{Y},\mathcal{X}):=\{\text{proper birational morphisms } f:\mathcal{Y}\to \mathcal{X}\\
  \text{such that } f|_X=\mathrm{id}_X \text{ holds}\}. 
  \end{cases}
  \]

For $\mathcal{X}, \mathcal{Y} \in B(X,x)$, it holds that
$\mathrm{Hom}(\mathcal{Y},\mathcal{X}) = \{\mathrm{1pt}\}$ or $\emptyset$
since $\mathcal{X}$ is separated.
In particular,  we can define a preorder on  $B(X,x)$ by 
$$\mathcal{X}\leq \mathcal{Y} :\Longleftrightarrow \mathrm{Hom}(\mathcal{Y},\mathcal{X}) \neq \emptyset.$$
Moreover,
the preorder makes
$B(X,x)$ a directed set. Indeed, for any two $\mathcal{X}, \mathcal{Y}\in B(X,x)$,
there exists a separated flat integral $k^\circ$-scheme $\mathcal{Z}$ of finite type such that there exist proper birational morphisms $f_{\mathcal{Z}, \mathcal{X}}:\mathcal{Z}\to \mathcal{X}$ and $f_{\mathcal{Z}, \mathcal{Y}}:\mathcal{Z}\to \mathcal{Y}$.
The scheme $\mathcal{Z}$ is obtained by taking the graph of a birational map $\mathcal{X} \dashrightarrow \mathcal{Y}$ in $\mathcal{X}\times_X \mathcal{Y}$ (resp. $\mathcal{X}\times_{k^\circ} \mathcal{Y}$) if $k=k^\circ$ (resp. $k\neq k^\circ$).
 If $k=k^\circ$, by taking an appropriate isomorphism, then we may assume that an induced morphism $f_\mathcal{Z} :\mathcal{Z}\to X$ by the universal property of $\mathcal{X}\times_X \mathcal{Y}$ satisfies  
$f_\mathcal{Z}^{-1}(U)=U$   and  
   $f_\mathcal{Z}|_{f_\mathcal{Z}^{-1}(U)}=\mathrm{id}_U$
     for some open    subscheme $U\subset X$.
Then $f_{\mathcal{Z}, \mathcal{X}}\in  \mathrm{Hom}(\mathcal{Z},\mathcal{X}) $ and $f_{\mathcal{Z}, \mathcal{Y}}\in  \mathrm{Hom}(\mathcal{Z},\mathcal{Y}) $ hold.
Further, $f_\mathcal{Z}:\mathcal{Z}\to X$ is a proper birational morphism, and the existence of $c_{\mathcal{Z}}(x)$ follows from the properness of $f_{\mathcal{Z}, \mathcal{X}}$. 
If $k\neq k^\circ$, then $\mathcal{Z}$ is proper over $k^\circ$ since $\mathcal{X}\times_{k^\circ} \mathcal{Y}$ is proper over $k^\circ$, and $\mathcal{Z}$ is also a model since $\mathcal{X}$ and $\mathcal{Y}$ are models.
In the same way as the case when $k=k^\circ$, we may assume that $f_{\mathcal{Z}, \mathcal{X}}\in  \mathrm{Hom}(\mathcal{Z},\mathcal{X}) $ and $f_{\mathcal{Z}, \mathcal{Y}}\in  \mathrm{Hom}(\mathcal{Z},\mathcal{Y}) $.
Hence, $\mathcal{Z}\in B(X,x)$ in either case.
It means that $B(X,x)$ is a directed set.

Note that
if $k=k^\circ$, blow-ups of $X$ are not models of $X$ in the sense of Definition \ref{model}.
That is the reason of our description of $B(X,x)$.

\begin{Thm}\label{union}
  Let $X$ be a variety over a non-Archimedean field $k$.
  For any $x\in X^{\rm val}$, define $B(X,x)$ as above.
  If $B(X,x)\neq \emptyset$, then it follows that
  \[\wt {\scrH(x)} =\bigcup_{\mathcal{X}\in B(X,x)} \kappa(c_{\mathcal{X}}(x)) \cong \ilim[\mathcal{X}\in B(X,x)] \kappa(c_{\mathcal{X}}(x)).\]
  Here, $\bigcup_{\mathcal{X}\in B(X,x)} \kappa(c_{\mathcal{X}}(x))$ just means the set-theoritic union of the images of canonical injections $\kappa(c_{\mathcal{X}}(x))\hookrightarrow \wt{\scrH(x)}$.
\end{Thm}

\begin{Lem}\label{unionn}
  In the same situation, we fix an element $\mathcal{X}\in B(X,x)$.  Here,
  we define $B(X,\mathcal{X},x)$ as a full subcategory of $B(X,x)$ consisting of
  all elements $\mathcal{X}'\in \mathrm{Ob}(B(X,x))$ such that $\mathrm{Hom}(\mathcal{X}',\mathcal{X})\neq \emptyset$.
  Then $B(X,\mathcal{X},x)$ is a cofinal directed set in $B(X,x)$.
 In particular,
  the following diagram commutes.
\[
  \xymatrix{
  \ilim[\mathcal{X}'\in B(X,\mathcal{X},x)] \kappa(c_{\mathcal{X}'}(x))
    \ar[r]^{\cong} \ar[d] & \ilim[\mathcal{Y}\in B(X,x)] \kappa(c_{\mathcal{Y}}(x)) \ar[d] \\
     \bigcup_{\mathcal{X}'\in B(X,\mathcal{X},x)}  \kappa(c_{\mathcal{X}'}(x))
     \ar[r]^{=} & 
     \bigcup_{\mathcal{Y}\in B(X,x)}  \kappa(c_{\mathcal{Y}}(x))
  }
\]
\end{Lem}
\begin{proof}[Proof of Lemma \ref{unionn}]
Take $\mathcal{Y}\in B(X,x)$.
Then we obtain $\mathcal{Z}\in B(X,\mathcal{X},x)$ such that $\mathrm{Hom}(\mathcal{Z},\mathcal{Y})\neq \emptyset$ since $B(X,x)$ is a directed set.
\end{proof}

\begin{proof}[Proof of Theorem \ref{union}]
The universal morphism
$$\ilim[\mathcal{X}\in B(X,x)] \kappa(c_{\mathcal{X}}(x)) \to \bigcup_{\mathcal{X}\in B(X,x)} \kappa(c_{\mathcal{X}}(x)) $$
is an isomorphism.
Hence, by Lemma \ref{unionn}, it suffices to show that
$$\wt {\scrH(x)} =\bigcup_{\mathcal{X}'\in B(X,\mathcal{X},x)} \kappa(c_{\mathcal{X}'}(x)).$$
Here,  $\mathcal{X}\in B(X,x)$ is fixed as Lemma \ref{unionn}. Suppose that an open affine subscheme $U={\rm Spec}A$ in $\mathcal{X}$
contains $x$.
By Lemma \ref{res}, we obtain
\[
\wt {\scrH(x)} \cong \wt{K(X)},
\]
where $\wt{K(X)}$ is defined by taking the residue field of
the valuation field $K(X)$
 with respect to $x$ as in Lemma \ref{res}.
For each $f\in\wt {\scrH(x)} \setminus \kappa\left( c_\mathcal{X}(x)\right)$,
we can take some $g,h\in A$ such that  $|h|_x \geq |g|_x \neq 0$ and $f=\ol{g/h} \in \wt {\scrH(x)} \cong \wt{K(X)}$.
Set $a=0$ (resp. $a\in k^\circ$ such that $|a|_x\in (0,1)$) if $k=k^\circ$ (resp. if $k\neq k^\circ$).
Then it holds that $|a^l|_x\leq |g|_x\leq |h|_x$ for some $l\in \Z_{>0}$ since $|g|_x>0$.
Consider an ideal $I=(g,h,a^l)$  of $A$.
Then, there is an ideal sheaf $\mathscr{I}$ on $\mathcal{X}$ such that $\wt{I}=\mathscr{I}|_U$
by \cite[p.126, Exercise 5.15]{H},
where $\wt{I}$ is the ideal sheaf on $U$ associated with $I$.
Now we consider a blow-up $\pi$ of $\mathcal{X}$ along $\mathscr{I}$. That is,
$\pi:\mathcal{X}':={\mathscr Bl}_{\mathscr{I}}\mathcal{X}\to \mathcal{X}.$ 
If $k\neq k^\circ$, 
this $\mathcal{X}'$ is a model of $X$.
Indeed, if  that's the case, $V(\mathscr{I})\cap X =\emptyset$ holds since $a^l\in k^\circ$ is a unit in $k(\neq k^\circ)$, which implies $\mathcal{X}'\times_{{\rm Spec} k^\circ} {\rm Spec}k \cong X$.
In both cases, $\mathcal{X}'\in B(X,x)$ holds. Now
$\mathcal{X}'$ has an open affine scheme $U'={\rm Spec}A[g/h, a^l/h]$.
Set $A':=A[g/h,a^l/h]$.
Since $f=\ol{g/h}\in\wt {\scrH(x)} \setminus \kappa\left( c_\mathcal{X}(x)\right)$, it holds that 
$|g/h|_x= 1$.
Further, $|a^l/h|_x\leq 1$ holds by definition.
Hence, it follows that $A'\subset R_x=\{f'\in \kappa (\pi_X(x))\ | \ |f'|_x\leq 1\}$. 
It implies $c_{\mathcal{X}'}(x)\in {\rm Spec}A'$.
Then we obtain
$\ol{g/h} \in A'/\mathfrak{p}_{c_\mathcal{X'}(x)} \subset \kappa(c_\mathcal{X'}(x))$.
Here, $\mathfrak{p}_{c_\mathcal{X'}(x)}$ is a prime ideal on $A'$ corresponding to $c_\mathcal{X'}(x)$.
Thus, it follows that $f=\ol{g/h}$ comes from
the
natural injection $\kappa(c_\mathcal{X'}(x))\hookrightarrow \wt {\scrH(x)}$.
\end{proof}

\begin{Def}\label{class}
  Let $X$ be a variety over a non-Archimedean field $k$.
  Denote by $X^*$ \emph{the set of all points $x$ in $X^\mathrm{val}$ such that
  $\wt{\scrH(x)}= \kappa (c_{\mathcal{X}}(x))$ holds
  for some $\mathcal{X}\in B(X,x)$.}

\end{Def}

\begin{Thm}\label{equal}
  Let $X$ be a variety over a non-Archimedean field $k$.
  Suppose that $x\in X^{\rm val}$ satisfies $B(X,x)\neq\emptyset$.
  Then the followings are equivalent.
  \begin{enumerate}
    \item \label{X*}
    $x\in X^*$
    \item \label{fg over model}
    $\wt{\scrH(x)}$ is finitely generated over $\kappa (c_{\mathcal{X}}(x))$
    as a field for any $\mathcal{X}\in B(X,x)$.
    \item \label{fg over k}
    $\wt{\scrH(x)}$ is finitely generated over $\wt{k}$
    as a field.
  \end{enumerate}
\end{Thm}
\begin{proof}
For $\mathcal{X}\in B(X,x)$, $c_\mathcal{X}(x)\in \mathcal{X}$ is on the fiber of the closed point of ${\rm Spec} k^\circ$.
Then we obtain a canonical injection $\wt{k}\hookrightarrow \kappa(c_\mathcal{X}(x))$.
Since $\mathcal{X}$ is a $k^\circ$-variety, $\kappa(c_\mathcal{X}(x))$ is a finitely generated over $\wt{k}$ as a field.
Therefore, it follows that \eqref{X*} implies \eqref{fg over k}.

Since $\wt{k} \hookrightarrow \kappa(c_\mathcal{X}(x)) \hookrightarrow \wt{\scrH(x)}$, it follows that
\eqref{fg over k} implies \eqref{fg over model}.

Hence, it suffices to show that \eqref{fg over model} implies \eqref{X*}.
Take $\mathcal{X}\in B(X,x)$ and suppose that
$$\wt{\scrH(x)}=\kappa(c_\mathcal{X}(x))(f_1,\dots ,f_n)$$ for some $f_1,\dots,f_n\in \wt{\scrH(x)}$. 
In a similar way as the discussion of Theorem \ref{union}, 
we obtain $\mathcal{X}_1 \in B(X,x)$ such that $\kappa(c_\mathcal{X}(x))(f_1)\subset \kappa(c_{\mathcal{X}_1}(x))$.
By repeating this discussion,
we finally obtain $\mathcal{X}_n \in B(X,x)$ such that $$\wt{\scrH(x)}=\kappa(c_\mathcal{X}(x))(f_1,\dots ,f_n)\subset \kappa(c_{\mathcal{X}_n}(x)).$$
It implies $\kappa(c_{\mathcal{X}_n}(x))=\wt{\scrH(x)}$. Hence, \eqref{fg over model} implies \eqref{X*}.
\end{proof}

In general, $X^{\rm val}\neq X^*$. The following  gives such an example.
\begin{Ex}[{\cite[Chap.3 II, footnote 12, p.864]{10.2307/1968864}}]
  Let $k(S)$ be the rational function field in a variable $S$ over a field $k$ equipped with the trivial valuation on $k(S)$.
  Set $X=\mathbb{A}_{k(S)}^2=\textrm{Spec}k(S)[T,U]$, where $T$ and $U$ are variables corresponding to coordinates of $X$.
  Let $\overline{k(S)}[[T]]$ be the ring of formal power series in $T$ whose coefficients are in the algebraic closure of $k(S)$.
  This $\overline{k(S)}[[T]]$ has the non-trivial $T$-adic valuation $\nu$ 
  and the valuation $\nu$ can be extended to the quotient field $\overline{k(S)}((T))$.
  Now we fix a compatible system $\{S^{\frac{1}{2^n}}\}_{n\in \Z_{>0}}$ and define $\varphi : k(S)[T,U]\to \overline{k(S)}((T))$ by
  \[S\mapsto S, T\mapsto T, U\mapsto \sum_{n\geq1}S^{\frac{1}{2^n}}T^n.\]
This $\varphi$ is injective. Hence the pullback of $\nu$ through $\varphi$ is also a valuation.
Denote this valuation by $x\in X^{\rm val}$, where the base field for construction of $X^\mathrm{an}, X^\mathrm{val}$ and $X^*$ should be $k(S)$ in this example.
 By construction, $x$ has the center in $X$, so that
$B(X,x)\neq \emptyset$.
Further, 
$\bigcup_{n\geq 1}k(S^{\frac{1}{2^n}})\subset\wt{\scrH(x)}$ follows from {\it loc}. {\it cit}.
It implies that $\wt{\scrH(x)}$ is not finitely generated over $k(S)(=\wt{k(S)})$.
Hence, it follows from  Theorem \ref{equal} that $x\notin X^*$.
\end{Ex}
\begin{Rem}
On the other hand, when $k$ is an algebraically closed non-Archimedean field,
it holds that $(\mathbb{A}_{k}^1)^{\rm val}=(\mathbb{A}_{k}^1)^*$ (cf.  \cite[Proposition 2.3]{BR10}).
\end{Rem}

In Theorem \ref{union}, the existence of a model (i.e. $B(X,x)\neq\emptyset$) is crucial.
If $X$ or $k^\circ$ satisfies certain conditions, then it is guaranteed that a model exists.
The followings are results concerning such conditions.
\begin{Prop}
  Let $X$ be a projective variety over a non-Archimedean field $k$.
  For any $x\in X^{\rm val}$, it holds that
$B(X,x)\neq \emptyset$.

\end{Prop}
\begin{proof}
  Since $X$ is projective, we obtain some projective model $\mathcal{X}$ of $X$ as Remark \ref{pro}.
  For any projective model $\mathcal{X}$, 
  it follows from the valuative criterion of properness that the model $\mathcal{X}$ has
  the center of $x$.
  Hence, it holds that $\mathcal{X}\in B(X,x)$.
\end{proof}
\begin{Prop}\label{model existence for proper}
  Let $X$ be a proper variety over a non-Archimedean field $k$.
  Assume that $k^\circ$ is a DVR (not a field).
  For any $x\in X^{\rm val}$, 
  it holds that  $B(X,x)\neq \emptyset$.

\end{Prop}
\begin{proof}
  Since $k^\circ$ is a DVR, then $k$ is finitely generated over $k^\circ$.
  Hence $X$ is an integral scheme of finite type over the noetherian integral scheme ${\rm Spec} k^\circ$.
Therefore we can take some proper $k^\circ$-variety $\mathcal{X}$ such that $X\hookrightarrow \mathcal{X}$ is an open immersion as $k^\circ$-scheme by Nagata compactification \cite{nagata1962}.
Then  $\mathcal{X}\to {\rm Spec} k^\circ$ is surjective by the valuative criterion of properness.
 Since $k^\circ$ is one dimensional, it follows that $\mathcal{X}\to {\rm Spec} k^\circ$ is flat.
 Now $X\hookrightarrow \mathcal{X}_k :=\mathcal{X}\times_{{\rm Spec}k^\circ} {\rm Spec}k$ is an open immersion.
 Since $\mathcal{X}\to {\rm Spec}k^\circ$ is proper, $\mathcal{X}_k\to {\rm Spec}k$ is also proper.
 These give the following commutative diagram.
 \[
   \xymatrix{
    X \ar[r]^{\rm open} \ar[rd]_{\rm proper} & \mathcal{X}_k \ar[d]^{\rm proper} \\
     & {\rm Spec} k
   }
 \]
It follows from the above diagram that $X \hookrightarrow \mathcal{X}_k$ is proper.
In particular, $X$ is open and closed in $\mathcal{X}_k$.
Since $\mathcal{X}_k$ and $X$ are integral, it holds that $X\cong \mathcal{X}_k$.
Hence $\mathcal{X}$ is a proper model of $X$.
For any proper model $\mathcal{X}$, it follows from the valuative criterion of properness
that the model $\mathcal{X}$ has the center of $x$.
Hence, it holds that $\mathcal{X}\in B(X,x)$.
\end{proof}

Now we consider
 a variety  $X$ over a non-Archimedean field $k$ under the assumption that $k^\circ$ is a DVR (not a field).
Then we can take a proper $k$-variety $Y$ containing $X$ as an open subscheme by Nagata compactification.
Once we fix such a $Y$, we can regard $x\in X^{\rm val}$ as an element of $Y^{\rm val}$ through the injection $X^{\rm an}\hookrightarrow Y^{\rm an}$.
 Then Proposition \ref{model existence for proper} 
 implies $B(Y,x)\neq \emptyset$.
 Therefore, it follows from Theorem \ref{union} that 
$$\wt {\scrH(x)} =\bigcup_{\mathcal{Y}\in B(Y,x)} \kappa(c_{\mathcal{Y}}(x)).$$

\section{$\wt{\scrH (x)}$ for quasi monomial valuations}

In this section, we consider quasi monomial valuations.
Before that, we give the definition of monomial valuations for two different setups as below (Definitions \ref{m1}, \ref{m2}).
After that, we introduce quasi monomial valuations (Definition \ref{aq}).
Note that Definition \ref{m2} and \ref{aq} are common (cf. \cite[\S 24.1.2]{Mat-text}), however, 
Definition \ref{m1} is not common.
In this section, we assume that the base field $k$ is a trivially valued field.

\begin{Def}[Monomial valuation on a polynomial ring]\label{m1}
Let $k$ be a field. Fix $n\in \Z_{>0}$ and
set $A:=k[X_1,\dots,X_n]$.
Then, a \emph{monomial valuation} $|\cdot|$ on $A$ is a valuation on $A$ defined as follows:
There are positive real numbers $r_1,\dots,r_n$ such that
for any
\[
f=\sum_{I\in \Z_{\geq 0}^n} a_IX^I \in k[X_1,\dots,X_n]\setminus \{0\},
\]
 $|\cdot|$ returns the following values.
\[
|f|:=\max_{a_I\neq 0}r^I,
\]
where $I=(i_1,\dots,i_n)$, $a_I\in k$, $X^I=X_1^{i_1}\cdots X_n^{i_n}$ and $r^I=r_1^{i_1}\cdots r_n^{i_n}$.
Besides, define $|0|:=0$.
\end{Def}

This monomial valuation can be regarded as an element $x$ of $({\rm Spec}A)^{\rm an}$, where
 we regard $k$ as a trivially valued field.
In the setting of Definition \ref{m1},
if $r_i\leq 1$ holds for all $i$, then the center of $x$ in ${\rm Spec}A$ exists since $A\hookrightarrow \scrH(x)^\circ$ follows.
 
Since this definition is too restrictive,
 we would like to define monomial valuations more generally.
Then we make use of Cohen's structure theorem.

\begin{Def}[Monomial valuation on a nonsingular point of variety]
\label{m2}
Let $X$ be a variety over a trivially valued field $k$. Let $p\in X$ be a nonsingular point.
(We do not assume that
$p\in X$ is a closed point.)
By using Cohen's structure theorem, if $p$ is not the generic point of $X$, then
there exist $m\in \Z_{> 0}$ and an injection
$\kappa(p)\hookrightarrow \widehat{\mathscr{O}_{X,p}}$ (which is not unique)
such that
for any system of algebraic coordinates $(f_1,\dots,f_m)$ at the point $p\in X$,
we obtain the following isomorphism as $\kappa(p)$-algebra to the ring of formal
power series.
\[
\widehat{\mathscr{O}_{X,p}}\cong \kappa (p)[[t_1,\dots,t_m]],
\]
where this isomorphism sends $f_i$ to $t_i$ and depends on the choice of the
embedding of the residue field $\kappa(p)\hookrightarrow \widehat{\mathscr{O}_{X,p}}$.
For simplicity, we regard $\kappa(p)$ as a subring of $\widehat{\mathscr{O}_{X,p}}$ by fixing the injection $\kappa(p)\hookrightarrow \widehat{\mathscr{O}_{X,p}}$.

In the above situation, a \emph{monomial valuation $|\cdot|$ on a nonsingular point} $p\in X$
is a valuation on $\mathscr{O}_{X,p}$ defined as follows:
There are positive real numbers $r_1,\dots,r_m$ that are less than $1$ such that
for each
\[
f=\sum_{I\in \Z_{\geq 0}^m} a_If^I \in \mathscr{O}_{X,p} \setminus \{0\}
\subset \kappa(p) [[f_1,\dots,f_m]],
\]
$|\cdot|$ returns the following values.
\[
|f|:=\max_{a_I\neq 0}|r^I|,
\]
where $I=(i_1,\dots,i_m)$, $a_I\in \kappa(p)$, $f^I=f_1^{i_1}\cdots f_m^{i_m}$
and $r^I=r_1^{i_1}\cdots r_m^{i_m}$.
Besides, define $|0|:=0$.
If $p$ is the generic point of $X$, then  a monomial valuation $|\cdot|$ on $p\in X$
is defined as the trivial valuation on $\mathscr{O}_{X,p}(=K(X))$.

\end{Def}

At a glance, this definition depends on the
choice of algebraic coordinates at $p\in X$,
their values and the embedding of the residue field.
However, it does not depend on the choice of the embedding of the
residue field (See \cite[Proof of Proposition 3.1]{jonsson2010valuations}).
In addition, the mononimal valuation $|\cdot |$ can be regarded as a point $x$ of
$X^{\rm an}$. 
In particular, 
the center of $x$ in $X$ exists and it holds that
$c_X(x)=p$.
When we consider valuations, it is not essential to fix a birational model.
Therefore, the following valuation is the most essential among three definitions.
\begin{Def}[Quasi monomial valuation]\label{triv qmo}
  \label{aq}
  Let $X$ be a variety over a trivially valued field $k$.
  A valuation $x\in X^{\rm val}$ is a \emph{quasi monomial valuation}
  if
 there is a proper birational morphism
  $f:Y\to X$ such that
  the valuation $x$ coincides with a monomial valuation on some nonsingular point  $q\in Y$
  in the sense of Definition \ref{m2}.
\end{Def}
Note that the center of a quasi monomial valuation on $X$ can be a singular point.
Originally, quasi monomial valuation is a valuation over a trivially valued field.
However, we can extend it to a valuation over a nontrivially valued field in some situations
(see Definition \ref{qmo}).

Now we state a key lemma for monomial valuations.

\begin{Lem}\label{mon}
  Let $x=|\cdot|$ be a monomial valuation on the polynomial ring
  $k[X_1,\dots,X_n]$ over a trivially valued field $k$ for some $n\in \Z_{>0}$. Suppose that $\sqrt{|\mathscr{H}(x)^\times |}\cong \Q^r$ for some $r\in \Z_{\geq 0}$.
Then, $\widetilde{\mathscr{H}(x)}$ is isomorphic to the rational function field
  in $n-r$ variables over $k$. 
\end{Lem}
\begin{proof}
Define a homomorphism $\varphi :\Z^n \to |\mathscr{H}(x)^\times|$
by
\[I=(i_1,\dots ,i_n) \mapsto |X^I|=|X_1^{i_1}\cdots X_{n}^{i_n}|.\]
First of all, we  see that $\varphi$ is surjective.
 Since $\mathscr{H}(x)$ is the completion of $k(X_1,\dots , X_n)$
 with respect to $x$,
 for any $f\in\mathscr{H}(x)$ and any $\varepsilon >0$,
 there exists $g \in k(X_1,\dots , X_n)$ such that $|f-g|<\varepsilon$.
For non-zero $f\in \mathscr{H}(x)$, we can take $\varepsilon >0$ such that $\varepsilon <|f|$ .
Then  $|f|=|g|$ holds
since $x$ is non-Archimedean.
Hence it holds that $|k(X_1,\dots,X_n)^\times|=|\mathscr{H}(x)^\times|$.
Since $x$ is a monomial valuation,
for $f=\sum a_IX^I\in k[X_1,\dots ,X_n]$, we have
$|f|=\max\{|X^I|\in\R|\ a_I\neq 0\}$. It implies $|k(X_1,\dots,X_n)^\times|= {\ima} \varphi$.
Therefore, $\varphi$ is surjective.
That is, the following sequence is exact.
\[0\to {\rm ker}\varphi \to \Z^n \to  |\mathscr{H}(x)^\times| \to 0\]
Since ${\rm ker}\varphi$ is
a submodule of the free $\Z$-module $\Z^n$,
${\rm ker}\varphi$ is a free $\Z$-module.
Furthermore, it holds that ${\rm ker}\varphi\cong \Z^{n-r}$
because of the above exact sequence
and the isomorphism $\sqrt{|\mathscr{H}(x)^\times|}\cong \Q^r$.
Then we  show that
\[\widetilde{\mathscr{H}(x)}\cong{\rm Frac}(k[{\ker}\varphi]),\]
where $k[{\rm ker}\varphi]$ is the group ring of ${\ker}\varphi$ over $k$.
Consider an isomorphism $$k[\Z^n] \cong k[X_1^\pm,\dots,X_n^\pm]$$ defined by $\Z^n \ni I \mapsto X^I$.
Then
a morphism $k[\ker \varphi ] \to k[X_1^\pm ,\dots,X_n^\pm]$ induced by the homomorphism $\ker \varphi \to \Z^n$ sends
$ I \in 
\ker\varphi$ to $X^I$.
In particular, it induces an injection
 $k[{\rm ker}\varphi]\hookrightarrow k(X_1,\dots ,X_n)^\circ$,
 where the right-hand side is the valuation ring of 
 $k(X_1,\dots ,X_n)$ with respect to $x$.
 Then the injection induces a morphism $k[{\rm ker}\varphi]\to \widetilde{\mathscr{H}(x)}$.
 Moreover, the morphism is injective 
  since $|X^I|=1$ for any $I\in \ker \varphi$.
Hence we have an injection
$${\rm Frac}(k[{\rm ker}\varphi])\hookrightarrow
 \widetilde{\mathscr{H}(x)}.$$
 Therefore we have to show  that the homomorphism is surjective.
By Lemma \ref{res}, we obtain
$\widetilde{\mathscr{H}(x)}\cong\widetilde{k(X_1,\dots ,X_n)}$.
Now we can write any element in $\widetilde{\mathscr{H}(x)}$
as $\overline{f}$,
 where $f\in k(X_1,\dots,X_n)^\circ$.
 For non-zero $\ol{f}\in \widetilde{\mathscr{H}(x)}$, let us choose $g,h\in k[X_1,\dots,X_n]$ such that $f=g/h$, where $g,h$ are non-zero.
If we write $g=\sum a_IX^I$, then,
for any $I\in \Z_{\geq 0}^n$ such that
$|a_IX^I|<|g|$,
\[\overline{\left( \frac{a_IX^I}{h}\right)}=0\]
holds since $|g|=|h|$.
Therefore we have
\[\overline{f}=\sum_{|a_IX^I|=|g|}\overline{\left(\frac{a_IX^I}{h}\right)}.\]
Note that,  since $|g|>0$, for  any $I\in \Z_{\geq 0}^n$ such that  $|a_IX^I|=|g|$,  $a_I$ is non-zero and $|X^I|=|g|=|h|$. 
Hence, it is enough to show that,
for each $I\in \Z_{\geq 0}^n$,  $|X^I|=|h|$ implies
$\overline{X^I/h}\in {\rm Frac}(k[{\rm ker}\varphi])$.
Assume $|X^I/h|=1$. It means that $\overline{X^I/h}\neq 0$.
Since ${\rm Frac}(k[{\rm ker}\varphi])$ is a field,
it is enough to show that $\overline{h/X^I} \in {\rm Frac}(k[{\rm ker}\varphi])$.
If we write $h=\sum b_JX^J$, then
\[\overline{\left(\frac{h}{X^I}\right)}=\sum_{|b_JX^J|=|X^I|}\overline{\left(\frac{b_JX^J}{X^I}\right)}=
\sum_{|X^{J-I}|=1}b_J\overline{X^{J-I}}\]
holds in the same way as above.
Hence, $\overline{h/X^I} \in {\rm Frac}(k[{\rm ker}\varphi])$ holds since
 $X^{J-I}\in k[{\rm ker}\varphi]$.
Therefore,
$\widetilde{\mathscr{H}(x)}\cong{\rm Frac}(k[{\rm ker}\varphi])$ holds.
In particular,
we obtain an isomorphism $
k[{\rm ker}\varphi]\cong k[\Z^{n-r}]
$
since ${\rm ker}\varphi \cong \Z^{n-r}$.
Hence, 
the assertion holds.
\end{proof}

Actually, Berkovich proves  that 
 $\widetilde{\mathscr{H}(x)}$ is isomorphic to the rational function field over $k$ (See \cite[Lemma 5.8]{Berk99}). 
 However, he does not compute its transcendental degree.
Lemma \ref{mon} enables us to compute the transcendental degree of $\widetilde{\mathscr{H}(x)}$.
For 
any variety $X$ over a trivially valued field $k$ and any $x\in X^{\rm val}$, it is known that the following inequality holds.
$$
\dim _\Q \sqrt{|\mathscr{H}(x)^\times |} + {\rm tr deg}_k  \widetilde{\mathscr{H}(x)}
\leq \dim X.
$$
This is called the \emph{Abhyankar inequality} (cf. 
\cite{ABH}).
We call $x$ an \emph{Abhyankar valuation} when the equality is achieved.
Lemma \ref{mon} implies that monomial valuations are Abhyankar valuations
since, in the setting of Lemma \ref{mon},
it holds that $\dim _\Q \sqrt{|\mathscr{H}(x)^\times |}=r$, ${\rm tr deg}_k  \widetilde{\mathscr{H}(x)}=n-r$
and $\dim X =n$.

\begin{Cor}\label{equal2}
  In the above situation, 
  set $X=\mathbb{A}_k^n$.
  If  the center of $x$ in $X$ exists,
then
  there exists a blow-up
  $\pi : X'\to X$ such that
  \[\wt {\scrH(x)} =\kappa(c_{X'}(x)).\]
\end{Cor}

\begin{proof}
  It follows from Lemma \ref{mon} and Theorem \ref{equal}.
\end{proof}

We can extend Corollary \ref{equal2}
to quasi monomial valuations as follows.

\begin{Thm}\label{smon}
  Let $X$ be a variety over a trivially valued field $k$ and
$x\in X^{\rm val}$ be a quasi monomial valuation.
   Then, $\widetilde{\mathscr{H}(x)}$ is finitely generated
   over $\kappa(c_X(x))$ as a field and
   there exists a blow-up
   $\pi : X'\to X$ such that
   \[\wt {\scrH(x)} =\kappa(c_{X'}(x)).\]
\end{Thm}

\begin{proof}
If $x$ is the trivial valuation on $K(X)$, then $X$ itself satisfies the condition.
Hence, from now on, we assume that $x$ is not the trivial valuation.
  By taking an appropriate blow-up, we may assume that $x\in X^{\rm val}$ is a monomial valuation on a nonsingular point $p\in X$.
  Set $\sqrt{|\mathscr{H}(x)^\times |}\cong \Q^r$
  and $\dim \mathscr{O}_{X,p}=m(>0)$.
Since $x\in X^{\rm val}$ is a monomial valuation on $p\in X$,
the valuation $x$, in particular, satisfies the following.

Use the same notation as Definition \ref{m2}.
  For some system of algebraic coordinates $(f_1,\dots,f_m)$
   at  $p\in X$ and some $r_1,\dots,r_m\in \R_{>0}$ that are less than 1,
   if $f\in \mathscr{O}_{X,p}\setminus \{0\}$ is of the form  
  \[
  f=\sum_{I\in \Z_{\geq 0}^m} a_If^I \in \mathscr{O}_{X,p}
  \subset \kappa(p) [[f_1,\dots,f_m]],\]
  then it holds that
  \[
  |f|_x=\max_{a_I\neq 0}|r^I|,
  \]
  where we regard $\kappa(p)$ as a subring of $\widehat{\mathscr{O}_{X,p}}$ by fixing an injection $\kappa(p)\hookrightarrow \widehat{\mathscr{O}_{X,p}}$.

  Set
  $A:=\kappa(p)[f_1,\dots,f_m]\subset \widehat{\mathscr{O}_{X,p}}$, $Y:={\rm Spec}A$
  and
  $y:=x|_A\in Y^{\rm val}$.
Then $A\subset \widehat{\mathscr{O}_{X,p}}$ is dense with respect to the valuation $x$.
By a similar argument to the latter half of the proof of Proposition \ref{Hx}, it follows that ${\rm Frac} (A)\subset {\rm Frac}(\widehat{\mathscr{O}_{X,p}})$ is dense  with respect to the valuation $x$.
Here, note that $\widehat{\mathscr{O}_{X,p}}$ is isomorphic to the completion of $\mathscr{O}_{X,p}$ with respect to  $x$.
Indeed, 
set $r:=\min_{i}r_i$,  $R:=\max_{i}r_i$
and the maximal ideal $\mathfrak{m}_p:=(f_1,\dots,f_m)$ of $\widehat{\mathscr{O}_{X,p}}$.
Then, since $0<r\leq R<1$, the inducing topology on $\mathscr{O}_{X,p}$
by $x$ is equivalent to the $\mathfrak{m}_p$-adic topology on $\mathscr{O}_{X,p}$.
It gives an injection $\widehat{\mathscr{O}_{X,p}} \hookrightarrow \scrH(x)$.
Since $\scrH(x)$ is a field, this injection factors through an injection ${\rm Frac}(\widehat{\mathscr{O}_{X,p}})\hookrightarrow\scrH(x)$.
Since $\scrH(x)\cong \widehat{K(X)}$, the image of ${\rm Frac}(\widehat{\mathscr{O}_{X,p}})$ by the injection  is dense.
Moreover, the image of ${\rm Frac}(A)$ by the injection is also dense.
Hence, $\scrH(y)=\scrH(x)$ holds.
  In particular, we obtain 
  $\wt {\scrH(y)}=\wt {\scrH(x)}$.
  Now $A$ is isomorphic to the polynomial ring over $\kappa(p)$.
  That is, $A\cong \kappa(p)[t_1,\dots,t_m]$.
  Then, we can regard $y$ as a monomial valuation on $A$
  in the sense of Definition \ref{m1}.
  Hence we can apply Lemma \ref{mon} to $A$.
  It implies that $\widetilde{\mathscr{H}(x)}$ is the rational function field
  in  $m-r$ variables over $\kappa(c_{Y}(y))$.
    In particular, $\wt{\scrH(y)}$ is finitely generated over $\kappa(c_{Y}(y))$
   as a field.
  By definition of $A$, it holds that $\kappa (c_Y(y)) = \kappa (p) =\kappa (c_X(x))$ as a subring of $\scrH(x)$.
  Hence,  $\kappa (c_Y(y)) =\kappa (c_X(x))$ also holds as a subring of $\wt{\scrH(x)}$.
It means that
     $\wt {\scrH(x)}$ is finitely generated over $\kappa(c_{X}(x))$
    as a field.
Finally, by Theorem \ref{equal}, we can construct a blow-up   $\pi : X'\to X$ such that
  \[\wt {\scrH(x)} =\kappa(c_{X'}(x)).\]
\end{proof}

\begin{Rem}\label{abh}
  In the above situation, we can see that
  a quasi monomial valuation $x\in X^{\rm val}$ is an Abhyankar valuation. Indeed,
we can reduce it to the case of monomial valuations since
   the equality ${\rm tr deg}_k ({\rm Frac} (A)) = \dim X$ holds and $y=x|_A$ is a monomial valuation on $A$.
  Actually, any Abhyankar valuation that admits the center on $X$ is a  quasi monomial valuation
  when the characteristic $k$ is $0$ (cf.  \cite[Theorem 1.1]{ASENS_2005_4_38_6_833_0},   \cite[Proposition 2.8]{10.2307/25099178}).
\end{Rem}

\begin{Cor}\label{rat}
In the same situation as Theorem \ref{smon},  suppose that $p$ is a
nonsingular closed point
whose residue field is $k$ and $x\in X^{\rm val}$ is a monomial valuation on the nonsingular point $p$.
Then   $\widetilde{\mathscr{H}(x)}$ is the rational function field over $k$.
\end{Cor}
\begin{proof}
Keep the notation introduced in the proof of Theorem \ref{smon}.
Since $\kappa(c_Y(y))$ is a $k$-algebra and $p\in X$ is a closed point,
the equality $\kappa(c_Y(y))=\kappa(c_X(x))=k$ follows from the discussion of Theorem \ref{smon}.
As we saw in the proof of Theorem \ref{smon}, $\widetilde{\mathscr{H}(x)}$ is the rational function field over $\kappa(c_{Y}(y))$.
Hence, the assertion holds.
\end{proof}
Theorem \ref{smon} shows that quasi monomial valuations are in $X^*$ (See Definition \ref{class}).
However, in general, there exists $x\in X^*$ such that $x$ is not a quasi monomial valuation.
\begin{Ex}[cf. {\cite[Example 1.7]{STEV}}]
Let $k$ be a trivially valued field, $k[X,Y]$ be the polynomial ring in $X,Y$ over $k$ and $k[[X]]$ be the
ring of formal power series in $X$ over $k$.
  Consider the homomorphism $k[X, Y]\hookrightarrow k[[X]]$ defined by sending $Y$ to some transcendental element over $k(X)$.
It induces an injection $k(X, Y)\hookrightarrow k((X))$.
Then, we denote by $x\in (\mathbb{A}_k^2)^{\rm val}$
 the pull back of the non-trivial discrete valuation on $k((X))$.
By definition, for any $f\in k[X,Y]$, it holds that $|f|_x\leq 1$.
It implies the existense of the center of $x$ in $\mathbb{A}_k^2$.
In particular, $B(x,\mathbb{A}_k^2)$ is non-empty.
For this $x$, it holds that $\wt{\scrH(x)}=k$ and ${\rm rank} ({\rm Im} |\ |_x)=1$.
In particular, $x\in (\mathbb{A}_k^2)^*$ follows from Theorem \ref{equal}.
On the other hand, if $x$ is a quasi monomial valuation, then $x$ is
an Abhyankar valuation. 
However,
 the Abhyankar equality ${\rm tr.deg}_k
\wt{\scrH(x)} =2- {\rm rank} ({\rm Im} |\ |_x)$ does not hold for this $x$.
Hence, $x$ is neither  an Abhyankar valuation nor a quasi monomial valuation.
\end{Ex}

\section{$\wt{\scrH (x)}$ for finite group action}
In this section, we consider a relation between the double residue field and finite group actions.

\begin{Def}
Let $X$ be a variety over a field $k$
and $G$ be a finite group.
We say that \emph{$G$ acts on $X$} if we fix a group homomorphism $G\to{\rm Aut}(X)$.
\end{Def}
For brevity, we identify $\sigma \in G$ with its image by the above homomorphism.
\begin{Def}
Let $X$ be a variety over a non-Archimedean field $k$ and $G$ be a finite group acting on $X$.
 Take a point $x\in X^{\rm val}$.
Then $x$ is \emph{$G$-invariant} if it satisfies the following.
\[
|f|_x=|\sigma^\sharp(f)|_x\ \ \ \ \  {}^{\forall}f\in K(X),{}^{\forall}\sigma\in G,
\]
where the ring isomorphism $\sigma^\sharp : K(X)\to K(X)$ is induced by the morphism $\sigma : X\to X$.
It means that $x=\sigma^{\rm an}(x)$ holds for all $\sigma \in G$, where
 the morphism $\sigma^{\rm an} : X^{\rm an} \to X^{\rm an}$ is induced by the morphism
  $\sigma : X\to X$ as in Fact \ref{un}.
\end{Def}
If $x\in X^{\rm val}$ is $G$-invariant, then
 $G$ acts on $\mathscr{H}(x)$ and $\widetilde{\mathscr{H}(x)}$.
 Indeed, for each $\sigma\in G$,
 $\sigma$ induces an isomorphism $$\sigma^\sharp: (K(X),|\cdot |_{\sigma^{\rm an}(x)})\to  (K(X),|\cdot |_x)$$
 between normed rings.
Since $x=\sigma^{\rm an}(x)$, it induces an automorphism  of $\mathscr{H}(x)$ as a complete valuation field.
In this way, $G$ acts on $\mathscr{H}(x)$.
Further, this action descends to $\widetilde{\mathscr{H}(x)}$ since the induced automorphism preserves the valuation $|\cdot |_x$ on $\mathscr{H}(x)$.

  Let $X$ be a variety over a trivially valued field $k$,
   $G$ be a finite  group acting on $X$ and
 $x\in X^{\rm val}$ be a $G$-invariant valuation that has the center  $p\in X$.

\begin{Lem}
  In the above situation,
   $p\in X$ is a fixed point of $G$.

\end{Lem}
\begin{proof}
Consider a morphism $\iota : {\rm Spec}K(X) \to X$
  induced by the identity map ${\rm id} : \mathscr{O}_{X,\eta}=K(X) \to K(X)$,
  where $\eta$ is the generic point of $X$.
Set $R_y:=\scrH(y)^\circ \cap K(X)$ for any $y\in X^{\rm val}$.
Since the center of $x$ is $p\in X$, there exists 
$\varphi: {\rm Spec} R_x \to X$ such that $\varphi$ makes the following diagram commmutative.

\[
  \xymatrix{
    {\rm Spec} K(X) \ar[r]^-{\iota} \ar[d]
    & X \ar[d] \\
    {\rm Spec} R_{x} \ar[ru]^\varphi \ar[r] & {\rm Spec}k
  }
\]
Here, the center
$p$ is given as the image of the closed point  of ${\rm Spec} R_x$ through $\varphi$.
Now we take $\sigma\in G$. Then $\sigma$ gives a $k$-isomorphism $\sigma:X\to X$.
Since
$x=\sigma^{\rm an}(x),$
it holds that $R_{\sigma^{\rm an}(x)}=R_x$  and $\sigma$ induces an automorphiosm
$ \sigma: {\rm Spec} R_x\to  {\rm Spec} R_x(={\rm Spec} R_{\sigma^{\rm an}(x)} )$. Note that
$\sigma$ sends the closed point of ${\rm Spec} R_x$ to itself.
Moreover, the following diagram is commutative.
\[
  \xymatrix{
    {\rm Spec} R_x \ar[r]^-{\varphi} \ar[d]_-{\sigma}
    & X \ar[d]^-{\sigma} \\
    {\rm Spec} R_{x}  \ar[r]^-\varphi & X
  }
\]
Hence, $p=\sigma (p)$ holds.
\end{proof}

\begin{Lem} \label{qnbd}
In the same situation as above, we can take some affine open set  $U\subset X$
such that $U$ is stable under the group actions of $G$ and $U$
contains $p\in X$ as a
fixed point of $G$.

\end{Lem}
\begin{proof}
Any $\sigma\in G$ corresponds with the $k$-isomorphism $\sigma:X\to X$.
Now we take some open affine neighborhood  $V\subset X$
of $p\in X$.
Then we set \[
U:=\bigcap_{\sigma\in G} \sigma^{-1}(V).
\]
Since $\sigma$ is an isomorphism, $\sigma^{-1}(V)$ is affine.
Moreover $p\in \sigma^{-1}(V)$ holds since $p$ is a fixed point by the above lemma.
Since $X$ is separated and $G$ is finite, 
$U$ is an affine neighborhood
of $p\in X$.
By the definition of $U$, it holds that  $U$ is stable under $G$.
Hence the assertion follows.
\end{proof}

This lemma implies that $G$ acts on the affine variety $U$.
Then
we can take the geometric quotient $\phi : U\to U/G$ for this
affine neighborhood $U$ of $p\in X$.
Let $A$ be a ring such that $U={\rm Spec}A$. 
Then the above morphism is given by the
inclusion $A^G\hookrightarrow A$, where $A^G$ is the invariant ring of $G$-actions on $A$.
Since $A$ is integral over $A^G$, $A^G$ is finitely generated over $k$ by Artin-Tate lemma.
In particular, 
$U/G$ is an affine variety.

\begin{Prop}\label{qfcn}
  Suppose that a finite group $G$ acts on an affine variety $X$ over a field $k$.
  Then, it holds that
  \[K(X)^G=K(X/G).\]
\end{Prop}
\begin{proof}
It suffices to show that $K(X)^G\subset K(X/G)$ since $K(X/G)\subset K(X)^G$ is trivial.
Set $X={\rm Spec}A$ and
take $f,g\in A$.
Here, assume that $g\neq 0$ and $f/g\in K(X)^G$.
For $f/g\in K(X)^G$, we consider
\[ h:=\prod_{\sigma\in G}\sigma(g).\]
Then it holds that $h\in A^G$ and $f/g\cdot h \in A^G$.
It implies that $$f/g=(f/g\cdot h)/h \in K(X/G).$$
Hence, the assertion holds.
\end{proof}

\begin{Thm}\label{quot}
  Let $X$ be a variety over  a trivially valued field $k$
  and $G$ be a finite group acting on $X$ whose order $|G|$ is relatively prime to the characteristic of $k$.
  Take $x\in X^*$(see Definition \ref{class}) and
assume that $x\in X^*$ is a $G$-invariant valuation.
Further, take $U$ as Lemma \ref{qnbd} and
let $\phi : U\to U/G$ be the geometric quotient.
 Denote by  $p\in X$  the center of $x$ and set
 $\phi ^{\rm an}(x)=y\in (U/G)^{\rm an}$.
   Then,
it follows that 
   $\widetilde{\mathscr{H}(x)}^G\cong \widetilde{\mathscr{H}(y)}.$
\end{Thm}

\begin{proof}
  In short, this proof is obtained by refining the proof of
Theorem \ref{union}.

By Lemma \ref{qnbd}, for this
 open affine neighborhood  $ U\subset X$
of $p\in X$, it holds that $U$ is stable under $G$.
Let $A$ be a ring such that $U={\rm Spec} A$.
Since $x\in X^*$,
it follows from Theorem \ref{equal} that
  \[\wt {\scrH(x)} =\kappa(c_{X}(x))(s_1,\dots ,s_r)\
  \text{ for \ some}\ r\in \Z_{>0} \text{ and } s_1,\dots,s_r\in\wt {\scrH(x)}.
  \]
  In a similar way as Theorem \ref{union}, for any $i\in \{1,\dots, r\}$, we set $s_i = \ol{f_i/g_i}$,
  where $f_i, g_i\in A$ with $|f_i|_x\leq |g_i|_x$ and $g_i\neq 0$.
  Now we may assume that $|f_i|_x=|g_i|_x=:r_i<1$.
  Indeed, if $r_i=1$, then 
$\ol{g_i}\neq 0$ in $A/c_X(x)$.
  Hence,
  $\ol{f_i/g_i}=\ol{f_i}/\ol{g_i}\in \fr (A/c_X(x)) =\kappa(c_{X}(x))$
  holds.
  Note that $r_i>1$ never holds since the center of $x$ is contained in $U={\rm Spec} A$.
  Let $n\in \Z_{>0}$ be the order of $G$.
Since $U$ is stable under $G$,   the action of $G$ on $X$ induces an action of $A$. That is, each $\sigma \in G$ gives
an
isomorphism $\sigma : A\to A$.

Now,
  we set
   $I_1=(f_1,g_1)\subset A$, which is an ideal of $A$.
For this $I_1$, we define $J_1$ as follows.
\[
J_1:=GI_1\cap B_{1} \subset A,
\]
where
\[GI_1:=\sum_{\sigma \in G} \sigma (I_1),\
B_{1}:=\{f\in A\ |\ |f|_x\leq r_1^n\}.
\]
Since $GI_1$ is a sum of ideals of $A$, $GI_1$ is an ideal of $A$.
Since $x$ is non-Archimedean and $A\subset R_x$,
 $B_1$ is also an ideal of $A$.
Therefore $J_1$ is also an ideal of $A$.
Since $x$ is $G$-invariant, $G$ acts on $J_1$.
Take $m_1\in \Z_{>0}$ and $a_{1 j}\in A$ for each $j\in\{1,\dots, m_1\}$ such that  $J_1=(a_{11},\dots,a_{1m_1})\subset A$.
Set
\[
h_1:=\prod _{\sigma\in G}\sigma (g_1) \in A.
\]
Then $h_1/g_1\in A$ holds.
In a similar way to Theorem \ref{union},
there is an ideal sheaf $\mathscr{J}_1$ on $X$ such that
$\wt{J_1}=\mathscr{J}_1|_U$.
   Now we consider the blow-up $\pi_1$ of $X$ along $\mathscr{J}_1$. 
   Set $X_1:={\mathscr Bl}_{\mathscr{J}_1}X$.
   Then we have
   $\pi_1:X_1\to X.$
Here,   $X_1$ has an open affine scheme $U_1={\rm Spec}A_1$,
   where
   \[
A_1:=A\left[\frac{a_{11}}{h_1},\dots,\frac{a_{1m_1}}{h_1}\right] .
   \]
Indeed, we can take this affine open set $U_1$ as follows:
We obtain the following diagram by the property of blow-up.
\[
  \xymatrix{
    {\mathscr Bl}_{\wt{J_1}}U \ar@{^{(}->}[r]_{\rm open} \ar[d]_{\pi_1} & X_1  \ar[d]^{\pi_1} \\
    U \ar@{^{(}->}[r]^{\rm open} & X
  }
\]
Hence,
$
U_1 := D_+(h_1) \subset {\mathscr Bl}_{\wt{J_1}}U
$
is also an open affine subscheme of $X_1$ that is of the form ${\rm Spec}A_1$.
Then, it holds that $A_1\subset R_x$, $G$ acts on $A_1$ and
\[
\frac{f_1}{g_1}=\frac{f_1\cdot (h_1/g_1)}{h_1}\in A_1.
\]
Indeed,   $f_1\cdot (h_1/g_1)\in J_1$ holds since $|f_1\cdot (h_1/g_1)|_x=|h_1|_x= r_1^n$.

In the same way,
we set
$$I_2:=(f_2,g_2)A_1, \
B_2:=\{f\in A_1 \ |\ |f|_x\leq r_2^n\},$$
$$
J_2:=GI_2\cap B_2 =(a_{21},\dots ,a_{2m_2}),\
h_2:=\prod_{\sigma \in G}\sigma (g_2) \in A \subset A_1
$$
for some $m_2\in \Z_{>0}$.
Moreover, take   an ideal sheaf $\mathscr{J}_2$ on $X_1$ such that
$\wt{J_2 }=\mathscr{J}_2 |_{U_1}$, and consider the following blow-up in the same way.
\[
\pi_2: X_2={\mathscr Bl}_{\mathscr{J}_2}X_1\to X_1.
\]
Then we can take an open affine subscheme $U_2={\rm Spec} A_2$ of $X_2$, where
\[
A_2:=A_1\left[\frac{a_{21}}{h_2},\dots,\frac{a_{2m_2}}{h_2}\right] .
\]
It implies that  $A_2\subset R_x$, $G$ acts on $A_2$ and
\[
\frac{f_2}{g_2}=\frac{f_2\cdot (h_2/g_2)}{h_2}\in A_2.
\]

Inductively,
we can construct the blow-up $\pi=\pi_r\circ \cdots \circ \pi_1 : X_r \to X$
and
take an open affine subscheme $U_r={\rm Spec} A_r$ of $X_r$, where
$$
A_r:=A_{r-1}\left[\frac{a_{r1}}{h_r},\dots,\frac{a_{rm_r}}{h_r}\right] .
$$
Then it holds that  $A_r \subset R_x$, $G$ acts on $A_r$ and
$f_i/g_i\in A_r$ for all $i$.

Now we can write $c_{X_r}(x)$ as $\mathfrak{m}_x\cap A_r\in U_r$. Here, recall that $\mathfrak{m}_x=\{f\in \kappa(\pi_X(x))\ | \ |f|_x<1\}$.
Hence,
$
\fr (A_r/ \mathfrak{m}_x\cap A_r) = \kappa (c_{X_r}(x)).
$
Since $\kappa (c_X(x))\subset \kappa (c_{X_r}(x))$
and $f_i/g_i\in \kappa (c_{X_r}(x))$ for each $i$,
it holds that \[
\kappa (c_{X_r}(x)) = \wt{\scrH(x)}.
\]
By Proposition \ref{qfcn}, it follows that
$
\fr (A)^G =\fr(A^G).
$
Hence it follows that $A_r^G\subset \fr (A)^G =\fr(A^G)$.
Further,
$A_r^G\subset R_y$ follows from $A_r^G\subset R_x$.
Now, we define $c_{X_r}^G(y)$ as the center of $y$ in $U_r/G$.
Then we obtain the following diagram.
\[
  \xymatrix{
    A_r^G \ar@{^{(}->}[r] \ar[d] & \scrH(y)^\circ \ar[d] \\
    \fr(A_r^G/c_{X_r}^G(y))\ar@{^{(}->}[r] & \wt{\scrH(y)}
  }
\]
Since $\scrH(y) \subset \scrH(x)$ and $\scrH(y)$ is stable under the action of $G$ on $\scrH(x)$,
it holds that $\wt{\scrH(y)}\subset \wt{\scrH(x)}^G$.
Hence it suffices to show $\wt{\scrH(x)}^G\subset\wt{\scrH(y)}$.
Since $|G|$ is relatively prime to the characteristic of $k$ by assumption,
it is well-known that a functor taking $G$-invariants is exact (cf. \cite[Corollary 4.4]{Web16}).
Hence,
$A_r^G/c_{X_r}^G(y)\cong (A_r/c_{X_r}(x))^G
$ holds.
Set $B:=A_r/c_{X_r}(x)$.
By the above diagram, $\fr(B^G)\subset \wt{\scrH(y)}$ holds.
Further,  $\fr (B)^G =\fr(B^G)$ follows from Proposition \ref{qfcn}.
Since $
\fr (B)=\wt{\scrH(x)}$, it holds that
$\wt{\scrH(x)}^G=\fr(B^G)$.
Hence, it follows that
$
\wt{\scrH(x)}^G\subset\wt{\scrH(y)}
$. 
\end{proof}

A $G$-invariant quasi monomial valuation $x\in X^{\rm an}$  
 satisfies the above condition. Indeed, $x\in X^*$ holds by Theorem \ref{equal} and Theorem \ref{smon}.
In this situation, it is natural to ask what the point $y=\phi ^{\rm an}(x)\in (U/G)^{\rm val}$ is like.
By definition of the valuation $y$, the center of $y\in (U/G)^{\rm an}$ is $\phi(c_X(x))$.
Further, $\sqrt{|\mathscr{H}(x)^\times |}=\sqrt{|\mathscr{H}(y)^\times |}$ and
${\rm tr deg}_k  \widetilde{\mathscr{H}(x)}={\rm tr deg}_k  \widetilde{\mathscr{H}(y)}$ hold.
In particular, $\sqrt{|\mathscr{H}(x)^\times |}=\sqrt{|\mathscr{H}(y)^\times |}$ holds by the following discussion.
Since $x$ is non-Archimedean, $|\mathscr{H}(x)^\times |=|K(U)^\times |$ holds, where the right hand side is an image of $K(U)^\times$ by the valuation $|\cdot|_x:K(U)\to \R$.
In the same way, $|\mathscr{H}(y)^\times |=|K(U/G)^\times |$ holds, where the right hand side is an image of $K(U/G)^\times$ by the valuation $|\cdot|_y:K(U/G)\to \R$.
For any $f\in K(U),$ we set $g:=\prod_{\sigma \in G}\sigma(f)\in K(U)$. Then $g\in K(U/G)$ holds.
Since $x$ is $G$-invariant,
$|g|=|f|^{n}$ holds, where $n\in \Z_{>0}$ is the order of $G$.
Hence,
$\sqrt{|K(U)^\times |}=\sqrt{|K(U/G)^\times |}$ holds, so that $\sqrt{|\mathscr{H}(x)^\times |}=\sqrt{|\mathscr{H}(y)^\times |}$ also holds.
It implies that $y\in  (U/G)^{\rm an}$ is an Abhyankar valuation
that has the center on $X$. By Remark \ref{abh},
$y\in (U/G)^{\rm an}$ is a quasi monomial valuation if the characteristic of $k$ is $0$.

\section{$\wt{\scrH (x)}$ over CDVF}

In $\S 4,5$, we only considered the case when the base field is
 a trivially valued field.
In this section, we consider the case when the base field $K$ is a CDVF.
Let $R$ be the valuation ring of $K$ and
 $k$ be the residue field of $K$.
Assume that the characteristic of $k$ is 0.
Then, Cohen's structure theorem implies an isomorphism $R\cong k[[\varpi]]$, where $\varpi$ is an uniformizing parameter of $R$.
Through this section, we regard $K$ as a non-Archimedean field
 equipped with a valuation  uniquely determined by $|\varpi|=\exp (-1)$ and we set
 $S:={\rm Spec}R$.
 
We prepare the following terminology as \cite{BFJ}.
\begin{Def}
$\mathcal{X}$ is an \emph{$S$-variety}
if it is a flat integral $S$-scheme of finite type.
We denote by $\mathcal{X}_0$ its central fiber
and by $\mathcal{X}_K$ its generic fiber.
\end{Def}

\begin{Def}

Let $\mathcal{X}$ be an $S$-variety.
An ideal sheaf $\mathscr{I}$ on $\mathcal{X}$ is \emph{vertical} if
it is co-supported on the central fiber.
A \emph{vertical blow-up} $\mathcal{X}'\to \mathcal{X}$ is the normalized blow-up
along a vertical ideal sheaf.
\end{Def}
Given an $S$-variety $\mathcal{X}$, let $\{ E_i\}_{i\in I}$ be
 the finite set of all irreducible components of its central fiber $\mathcal{X}_0$.
 For each non-empty subset $J\subset I$, we set
 \[E_J:=\bigcap_{j\in J}E_j.
 \]
 Here, we endow each $E_J$ with the reduced scheme structure.

\begin{Def}
Let $\mathcal{X}$ be an $S$-variety.
$\mathcal{X}$ is \emph{SNC} if it satisfies the following.
\begin{enumerate}
  \item the central fiber $\mathcal{X}_0$ has simple normal crossing support,

  \item $E_J$ is irreducible (or empty) for each non-empty subset $J\subset I$.
\end{enumerate}

\end{Def}

Condition (1) is equivalent to that the following two conditions holds. First, $\mathcal{X}$
 is regular.
 Given a point $\xi\in \mathcal{X}_0$, let $I_\xi \subset I$ be the set of
  indices $i\in I$ for which $\xi \in E_i$, and we pick
  a local equation $z_i\in \mathscr{O}_{\mathcal{X},\xi}$
  of $E_i$ at $\xi$ for each $i\in I_\xi$.
  Then we also impose that $\{ z_i \ |\ i\in I_\xi\}$ can be completed to
   a regular system of parameters of $\mathscr{O}_{\mathcal{X},\xi}$.

Condition (2)
  is not imposed in the usual definition of a simple normal crossing
 divisor. However, it can always be achieved from (1) by further blow-up along components of the possibly non-connected
  $E_J$'s.

We denote by ${\rm Div}_0(\mathcal{X})$ \emph{the group of
vertical Cartier divisors} on $\mathcal{X}$.
When $\mathcal{X}$ is normal,
 ${\rm Div}_0(\mathcal{X})$ becomes a free $\Z$-module of finite rank.

\begin{Fact}[cf.  {\cite[Theorem 1.1]{TEMKIN2008488}}]\label{Tem}
For any $S$-variety $\mathcal{X}$ with smooth generic fiber,
there exists a vertical blow-up $\mathcal{X}'\to\mathcal{X}$ such that
$\mathcal{X}'$ is SNC.

\end{Fact}

\begin{Def}
Let $X$ be a smooth connected projective $K$-analytic space
in the sense of Berkovich.
An  $S$-variety $\mathcal{X}$ is a \emph{model} of $X$ if it is a normal and projective
  $S$-variety together with the datum of an isomorphism $\mathcal{X}_K^{\rm an}\cong X$.
\end{Def}

In the above setting, 
for some smooth projective $K$-variety $Y$,
we can identify $X$ with $Y^{\rm an}$ by \cite[Proposition 3.3.23]{Berk90}.
Moreover,
 there is a model of $X$.
Indeed, given an embedding of $Y$ into a suitable
projective space $\mathbb{P}_K^m$,
we can take $\mathcal{X}$ as the normalization of
the closure of $Y$ in $\mathbb{P}_S^m$.
Then $\mathcal{X}$ is a model of $X$.
Note that $\mathcal{X}$ is also a model of $\mathcal{X}_K$ in the sense of Definition \ref{model}.
Now we assume the existence of two nontrivial Grothendieck universes $\mathcal{U}$ and $\mathcal{V}$ such that $\mathcal{X}\in \mathcal{U}\in \mathcal{V}$.
We denote by $\mathcal{M}_X$ \emph{the class of $\mathcal{U}$-small models} of $X$, where 
the $\mathcal{U}$-smallness means being an element of $\mathcal{U}$.
By the above discussion, it follows that $\mathcal{M}_X$ is nonempty and $\mathcal{V}$-small, which means $\mathcal{M}_X\in \mathcal{V}$.
In addition, 
it follows from a similar discussion to the case of $B(X,x)$ that $\mathcal{M}_X$ becomes a
 directed set by declaring $\mathcal{X}'\geq \mathcal{X}$
if there exists a 
proper birational morphism
$\mathcal{X}'\to \mathcal{X}$ whose restriction to generic fibers is an isomorphism that is compatible with the structure of models.

For any model $\mathcal{X}$ of $X$ and any $x\in X$,
we can define the center $c_{\mathcal{X}}(x)$ of $x$ in the same way as before.
That is, we consider the following diagram.
 \[
   \xymatrix{
     {\rm Spec}\ \kappa (\pi_{\mathcal{X}_K} (x')) \ar[r] \ar[d] & \mathcal{X} \ar[d] \\
     {\rm Spec} R_{x'} \ar[r] \ar@{..>}[ru] & S}
 \]
 Here, 
 we use the same notation as in \S  \ref{center} and
 $x'$ is the image of $x$ in $\mathcal{X}_K^{\rm an}$.
Then $c_\mathcal{X}(x)\in \mathcal{X}$ is obtained by the image of the closed point
of ${\rm Spec}R_{x'}$.
Further, 
 since the center is given as the image of the closed point
of ${\rm Spec}R_{x'}$ and the above diagram commutes,
it holds that
$c_\mathcal{X}(x)\in \mathcal{X}_0$.

Let $\mathcal{X}$ be an SNC model of $X$.
We can write the central fiber as \[
\mathcal{X}_0=\sum_{i\in I} m_iE_i,
\]
where $(E_i)_{i\in I}$
are irreducible components.
Then, it follows that \[
{\rm Div}_0(\mathcal{X}) =\bigoplus_{i\in I}\Z E_i .
\]
Set $
{\rm Div}_0(\mathcal{X})^*_\R:={\rm Hom}( {\rm Div}_0(\mathcal{X})
,\Z )\otimes_\Z \R.$
Denote by $E_i^*$ the dual element of $E_i$ and  set
\[e_i :=\frac{1}{m_i}E_i^*.\]

For each $J\subset I$ such that $E_J\neq \emptyset$,
let $\hat{\sigma}_J \subset {\rm Div}_0(\mathcal{X})^*_\R$
be a simplicial cone defined by
\[\hat{\sigma}_J := \sum_{j\in J}\R_{\geq 0}e_j.
\]
Fix the basis of ${\rm Div}_0(\mathcal{X})^*_\R$ as above.
That is, $s=(s_j)\in \hat{\sigma}_J$ means $s=\sum s_j e_j$.
These cones naturally defines a fan $\hat{\Delta}_{\mathcal{X}}$
in ${\rm Div}_0(\mathcal{X})^*_\R$.

Define the \emph{dual complex} of $\mathcal{X}$ by
\[
\Delta_{\mathcal{X}}:= \hat{\Delta}_{\mathcal{X}}
\cap \{ \left< \mathcal{X}_0,\cdot\right> =1\} ,
\]
where $\left< \cdot,\cdot\right>$ is the natural bilinear
form on ${\rm Div}_0(\mathcal{X})^*_\R$.
Each $J\subset I$ such that $E_J\neq\emptyset$ corresponds to a simplicial face
\[
\sigma_J:= \hat{\sigma}_J \cap \{ \left< \mathcal{X}_0,\cdot\right> =1\}
= {\rm Conv}\{e_j\ |\ j\in J\}
\]
of dimension $|J|-1$ in $\Delta_{\mathcal{X}}$,
where ${\rm Conv}$ denotes the convex hull.
Then we can define a structure of
a simplicial complex on $\Delta_{\mathcal{X}}$ such that,
for two subsets $J,L$ of $I$, $\sigma_J$ is a face of $\sigma _L$ if and
only if $J\supset  L$.

We denote by $\mathcal{M}_X'$ \emph{the full subcategory consisting of SNC models} of $X$ in $\mathcal{M}_X$.
By Fact \ref{Tem}, $\mathcal{M}_X\neq\emptyset$ implies $\mathcal{M}_X'\neq \emptyset$.
Besides, since $\mathcal{M}_X$ is directed,
$\mathcal{M}_X'$
becomes a directed set by Fact \ref{Tem}.
For two models
 $\mathcal{X}',\mathcal{X}\in \mathcal{M}_X'$, the binary relation
 $\mathcal{X}'\geq \mathcal{X}$ induces a natural map
$\Delta_{\mathcal{X}'}\to\Delta_{\mathcal{X}}$.
Hence, 
\[
\lim_{\substack{\longleftarrow \\ \mathcal{X}\in \mathcal{M}_X'}}\Delta_{\mathcal{X}}
\]
 is well-defined.
The following, which is a highly suggestive result,
 is stated in \cite{Kontsevich2006}.
Furthermore, the proof is written in \cite{BFJ}.

\begin{Fact}[{\cite[Corollary 3.2]{BFJ}}]\label{BFJ}
In the above situation, we obtain the following homeomorphism.
\[
X\cong
\lim_{\substack{\longleftarrow \\ \mathcal{X}\in \mathcal{M}_X'}}\Delta_{\mathcal{X}}
.
\]

\end{Fact}

 Now, for each non-empty subset $J\subset I$,
the intersection $E_J:=\cap_{j\in J}E_j$ is either empty
 or a smooth irreducible $k$-variety.
Let $\xi _J$ be a generic point of $E_J$ if $E_J\neq\emptyset$.
For each $j\in J$ we can choose a local equation
 $z_j\in \mathscr{O}_{\mathcal{X},\xi _J}$, such that $(z_j)_{j\in J}$
 is a regular system of parameters of $\mathscr{O}_{\mathcal{X},\xi _J}$
  because of the SNC condition.
 Since the valuation ring $R$ contains
 the residue field $k$, 
the completion  $\widehat{\mathscr{O}}_{\mathcal{X},\xi _J}$ of  $\mathscr{O}_{\mathcal{X},\xi _J}$ also contains the field $k$. Then we can also apply Cohen's structure theorem to   $\widehat{\mathscr{O}}_{\mathcal{X},\xi _J}$. Hence,
  after taking a field of representatives of $\kappa(\xi _J)$,
  we obtain that
  \[\iota : \widehat{\mathscr{O}}_{\mathcal{X},\xi _J}\cong \kappa(\xi _J)[[t_j,j\in J]]\]
defined by $\iota (z_j)= t_j$ for each $j\in J$.

\begin{Def}\label{qmo}
Under this setting,
 $x\in X$ is said to be a \emph{quasi monomial valuation} if there exist
  an SNC model $\mathcal{X}$ of $X$ and $s=(s_j)\in \sigma _J\subset \Delta_\mathcal{X}$
  such that $x$ is a valuation on $\mathscr{O}_{\mathcal{X},\xi _J}$ given as the restriction of the   following valuation $|\cdot |$ on $\widehat{\mathscr{O}}_{\mathcal{X},\xi _J}$.
  For any non-zero $f \in \widehat{\mathscr{O}}_{\mathcal{X},\xi _J}$ of the form
  $$f= \sum_{\alpha\in {\Z}_{\geq 0}^{|J|}}c_\alpha z^\alpha \in \widehat{\mathscr{O}}_{\mathcal{X},\xi _J}\setminus \{ 0\},$$
 the valuation $|\cdot |$ is defined by
$$|f|:=\max_{c_\alpha\neq 0}\exp(-s)^\alpha:=\max_{c_\alpha\neq 0} \left( \prod_{j\in J} \exp(-s_j)^{\alpha_j} \right),$$
where $\alpha=(\alpha_j)_{j\in J}$, $z^\alpha:=\prod_{j\in J} z_j^{\alpha_j}$ and
each $c_\alpha$ is either zero or
a unit of $\widehat{\mathscr{O}}_{\mathcal{X},\xi _J}$ such that $\iota (c_\alpha ) \in \kappa (\xi_J) $.
Besides, define $|0|:=0$.
\end{Def}
  \begin{Rem} \label{dif qmo}
We already defined quasi monomial valuation in Definition \ref{aq}.
The difference between the two definitions is that the above `quasi monomial
valuation' is not only a quasi monomial valuation in the sense of Definition \ref{aq},
but also an extension of the equipped valuation on $K$ and has a `good' center.

For the remainder of this paper, quasi monomial valuations will be in the sense of
Definition \ref{qmo}.
\end{Rem}

Our definition above is slightly a priori different from the original one
in
\cite{BFJ} though it is still equivalent.
Denote by $X^{\rm qm}$ \emph{the set of quasi monomial valuations of $X$}.

From now on, we list a few properties of quasi monomial valuations.

\begin{Fact}[{\cite[Corollary 3.9]{BFJ}}]
  $X^{\rm qm}$ is dense in $X$.
\end{Fact}
\begin{Fact}[cf. {\cite[Definition 3.7 and $\S 3.3$]{BFJ}}]
Denote by $\Delta_{\mathcal{X}}'$ the inverse image of $\Delta_{\mathcal{X}}$ through
the homeomorphism in Fact \ref{BFJ}.
Then, it holds that

\[
X^{\rm qm} = \bigcup_{\mathcal{X}\in \mathcal{M}_X'} \Delta_{\mathcal{X}}'.
\]

\end{Fact}
This is the definition of quasi monomial valuations in \cite{BFJ}.

\vspace{0.1in}

We prove the following property of quasi monomial valuations as an application of the discussion in Theorem \ref{smon}.

\begin{Thm}\label{quasi}
  Let $X$ be a smooth connected projective $K$-analytic space.
If $x$ is a quasi monomial valuation, then there exists an SNC model $\mathcal{X}$ of $X$ such that
 $\wt {\scrH(x)} =\kappa(c_{\mathcal{X}}(x))$.

\end{Thm}
\begin{proof}

This proof is essentially given by Theorem \ref{union}.
  Since $x$ is a quasi monomial valuation, we can take an SNC model $\mathcal{X}$ such
  that $x$ gives a monomial
  valuation on $\widehat{\mathscr{O}}_{\mathcal{X},c_\mathcal{X}(x)}$ as above.
   Now we construct a desirable vertical blow-up
   $\pi : \mathcal{X}'\to \mathcal{X}$ such that
   $\wt {\scrH(x)}=\kappa(c_\mathcal{X'}(x))$
   by refining
the construction of Theorem \ref{smon}.
Here, it follows from Theorem \ref{smon}
that
\[\wt {\scrH(x)} =\kappa(c_{\mathcal{X}}(x))(f_1,\dots ,f_n)\ \
{\rm for \ some}\ n\in \Z_{>0}, f_1,\dots,f_n\in\wt {\scrH(x)}.
\]
We may assume that $f_i\neq 0$ for all $i=1,\dots, n$.
Set $f_i = \ol{g_i/h_i}$,
where $g_i,h_i\in \mathfrak{m}_{c_{\mathcal{X}}(x)} \subset
\mathscr{O}_{\mathcal{X},c_{\mathcal{X}}(x)}$ 
with $|g_i|_x\leq |h_i|_x$
and $h_i\neq 0$ in the same way as the discussion of
 Theorem \ref{quot}.
Here, $\mathfrak{m}_{c_{\mathcal{X}}(x)}$ denotes the maximal ideal of $\mathscr{O}_{\mathcal{X},c_{\mathcal{X}}(x)}$.
The uniformizing parameter $\varpi$ of $R$ gives 
a
local equation 
$\varpi \in \mathscr{O}_{\mathcal{X},c_{\mathcal{X}}(x)}$  of the central fiber $\mathcal{X}_0$ at $c_{\mathcal{X}}(x)$.
Since $|\varpi|_x<1$ and $|g_i|_x=|h_i|_x<1$, we see
$|\varpi^l|_x \leq |g_i|_x=|h_i|_x$ for some
$l \in \mathbb{N}$. 
By taking a sufficiently large $l$, we can choose $l$ to be independent of the choice of $i$.

Consider the blow-up along the closed subscheme $V(\varpi ^l, g_1,h_1)$ in some
neighborhood $U={\rm Spec}A$  of $c_{\mathcal{X}}(x)$.
Since $V(\varpi ^l, g_1,h_1)$ is also a closed subscheme of $l\mathcal{X}_0$,
the defining ideal sheaf of $V(\varpi ^l, g_1,h_1)$ extend to some defining ideal sheaf
on $\mathcal{X}$ that contains the defining ideal sheaf of $l\mathcal{X}_0$ on $\mathcal{X}$
 in the same way as Theorem \ref{union}.
Then the blow-up along this ideal sheaf can be regarded as an SNC model
after
taking a further blow-up by \cite[Theorem 1.1]{TEMKIN2008488}.
Denote by $\pi_1 : \mathcal{X}^1 \to \mathcal{X}$ this blow-up.
Then we see $f_1 \in \kappa(c_{\mathcal{X}^1}(x))$. Indeed,
we can take $U^1={\rm Spec} A^1$ as the affine neighborhood of $c_{\mathcal{X}^1}(x)$
such that $g_1/h_1 \in A^1$ since $\pi_1|_U$ factors through the blow up along $V(\varpi ^l, g_1,h_1)$ of $U={\rm Spec}A$.
In particular, $c_{\mathcal{X}}(x) \in U$ is lifted to $U^1$ through the following diagram.
\[
\xymatrix{
\pi_1^{-1}(U)\supset U^1 ={\rm Spec} A^1\ar[rr]^-{\pi_1}\ar[drr]_-{}& & U={\rm Spec}A\\
& &{\rm Spec}A[\varpi^l/h_1, g_1/h_1]\ar@{}[ul]|{\ \ \ \ \circlearrowright}\ar[u]&
}
\]
Hence, $f_1 \in \kappa(c_{\mathcal{X}^1}(x))$ holds.

In the same way as above, we construct the vertical blow-up
$$\pi_{i+1}: \mathcal{X}^{i+1}\to \mathcal{X}^i$$
with respect to $V(\varpi ^l, g_{i+1},h_{i+1})$ inductively.
Then, it follows that
\[
\kappa(c_{\mathcal{X}^n}(x)) = \wt {\scrH(x)}.
\]

  In this way, we construct desirable vertical blow-up
  \[
  \pi (=\pi_n\circ\cdots \circ \pi_1) : \mathcal{X}'(=\mathcal{X}^n )\to \mathcal{X}.
  \]
Then $\mathcal{X}'$ is an SNC-model of $X$.
Therefore,
the assertion follows.
\end{proof}

\bibliographystyle{amsalpha}

\end{document}